\documentclass[review,hidelinks,onefignum,onetabnum]{siamart220329}

\usepackage{amsmath}
\usepackage{amsfonts}
\usepackage{lipsum}
\usepackage{graphicx}
\usepackage{epstopdf}
\usepackage{algorithmic}
\usepackage{braket,amsopn} 
\DeclareMathOperator{\Range}{Range}
\usepackage{array} 
\usepackage[caption=false]{subfig}
\usepackage{graphicx,epstopdf} 
\usepackage[caption=false]{subfig}

\usepackage{algorithmic} 
\Crefname{ALC@unique}{Line}{Lines}

\usepackage{bm}

\usepackage{enumerate}
\usepackage{diagbox}

\ifpdf
  \DeclareGraphicsExtensions{.eps,.pdf,.png,.jpg}
\else
  \DeclareGraphicsExtensions{.eps}
\fi

\newsiamremark{remark}{Remark}
\newsiamremark{hypothesis}{Hypothesis}
\crefname{hypothesis}{Hypothesis}{Hypotheses}
\newsiamthm{claim}{Claim}
\newsiamremark{example}{Example}

\headers{SVD of Dual Matrices}{Tong Wei, Weiyang Ding, and Yimin Wei}

\title{Singular Value Decomposition of Dual Matrices and its Application to Traveling Wave Identification in the Brain
\thanks{Submitted to the editors DATE.
\funding{This work is partially supported by Science and Technology Commission of Shanghai Municipality (No. 23ZR1403000, 20JC1419500, 2018SHZDZX0). Y. Wei is supported by the National Natural Science Foundation of China under grant 12271108, Innovation Program of Shanghai Municipal Education Committee, and Medical Engineering Joint Fund of Fudan University.}
}
}

\author{Tong Wei\thanks{Institute of Science and Technology for Brain-Inspired Intelligence, Fudan University, Shanghai, P.R. China. (\email{20110850008@fudan.edu.cn, weitong0802kun@163.com}).}
\and Weiyang Ding\thanks{Corresponding author, Institute of Science and Technology for Brain-Inspired Intelligence and MOE Frontiers Center for Brain Science, Fudan University; Shanghai Center for Brain Science and Brain-Inspired Technology, Shanghai, P.R. China; Key Laboratory of	Computational Neuroscience and Brain-Inspired Intelligence (Fudan University), Ministry of Education, P.R. China; Zhangjiang Fudan International Innovation Center. (\email{dingwy@fudan.edu.cn}).}
\and Yimin Wei\thanks{School of Mathematical Sciences and Key Laboratory of Mathematics for Nonlinear Sciences, Fudan University, Shanghai, P.R. China. 
(\email{ymwei@fudan.edu.cn}).}
}

\usepackage{amsopn}

\ifpdf
\hypersetup{
  pdftitle={Singular Value Decomposition of Dual Matrices and its Application to Traveling Wave Identification in the Brain},
  pdfauthor={Tong Wei, Weiyang Ding, and Yimin Wei}
}
\fi

\begin{document}

\maketitle

\begin{abstract}
Matrix factorizations in dual number algebra, a hypercomplex system, have been applied to kinematics, mechanisms, and other fields recently.
We develop an approach to identify spatiotemporal patterns in the brain such as traveling waves using the singular value decomposition of dual matrices in this paper.
Theoretically, we propose the compact dual singular value decomposition (CDSVD) of dual complex matrices with explicit expressions as well as a necessary and sufficient condition for its existence.
Furthermore, based on the CDSVD, we report on the optimal solution to the best rank-$k$ approximation under a newly defined quasi-metric in the dual complex number system.
The CDSVD is also related to the dual Moore-Penrose generalized inverse.
Numerically, comparisons with other available algorithms are conducted, which indicate less computational costs of our proposed CDSVD.
In addition, the infinitesimal part of the CDSVD can identify the true rank of the original matrix from the noise-added matrix, but the classical SVD cannot.
Next, we employ experiments on simulated time-series data and a road monitoring video to demonstrate the beneficial effect of the infinitesimal parts of dual matrices in spatiotemporal pattern identification.
Finally, we apply this approach to the large-scale brain fMRI data, identify three kinds of traveling waves, and further validate the consistency between our analytical results and the current knowledge of cerebral cortex function.
\end{abstract}

\begin{keywords}
dual matrices, compact dual SVD, low-rank approximation, dual Moore-Penrose generalized inverse, traveling wave identification, brain dynamics, time-series data analysis
\end{keywords}

\begin{MSCcodes}
15A23, 15B33, 65F55
\end{MSCcodes}

\section{Introduction}
Time-series analysis is commonly applied to uncover the underlying causes of trends or systemic patterns over time. 
Classical singular value decomposition (SVD) or principal component analysis (PCA) \cite{PCA} sufficiently explains a high-dimensional time-series process by low-dimensional principal components through an orthogonal transformation.
Thus, only a tiny portion of uncorrelated components statistically carries a great deal of initial information, and the importance of each component is decreasing in turn.
However, although these principal components are uncorrelated with each other, they are not necessarily independent.
Since each component is rank-1 and has no specific physical meaning, we attempt to find combinations of different components with certain physical meanings, such as traveling waves, by augmenting additional, e.g. temporal, information.

A natural idea is to 
augment the original data with differential terms concerning time by using a hypercomplex system.
The property that the square of the infinitesimal unit is equal to zero for dual numbers is precisely the same as the characteristic of the differential term.
Therefore, in this paper, we aim to explore the relationship, such as traveling waves, between principal components by revealing the SVD of dual matrices.

Dual numbers employed as a novel data structure are ubiquitous across a broad range of fields from kinematics and statics to dynamics \cite{fischer2017dual}.
In 1873, William Clifford first proposed dual numbers \cite{clifford1871preliminary}, which differ from real numbers in that they have extra infinitesimal parts.
Actually, infinitesimal parts of dual numbers can be regarded as perturbations of their standard parts, as well as derivatives concerning time for time-series data.
In addition, Eduard Study \cite{study1923grenzfall} applied dual numbers to measure the relative position of skew lines in three-dimensional space and defined a dual angle, whose standard and infinitesimal parts represent the angle and distance between lines respectively.
Later, dual algebra was further developed and has been widely used in robotics \cite{gu1987dual,pradeep1989use}, kinematic analysis \cite{angeles1998application,pennestri2009linear}, and spatial mechanisms \cite{cheng1996dual,cheng1997dual}.
In this case, more researchers pay attention to the fundamental theoretical properties of dual matrices, including the SVD \cite{gutin2021generalizations}, eigenvalues \cite{qidualquaternion}, low-rank approximation \cite{qi2022low}, dual Moore-Penrose generalized inverse \cite{angeles2012dual,qi2023moore,pennestri2018moore,udwadia2021dual,udwadia2021does,udwadia2020all,wang2021characterizations}, linear dual least squares \cite{leastsquares}, QR decomposition \cite{pennestri2007linear}, dual rank decomposition \cite{wang2022dual}, etc. 

Despite the immense ongoing efforts in studying properties and developing algorithms for the SVD of dual matrices, the intuitive explicit expressions are lacking and algorithms need urgent improvement. 
Some existing methods are intended for dual matrices with particular structures and suffer from structural redundancy as well as expensive costs. For instance, Pennestrì et al. \cite{pennestri2018moore} obtained the SVD of dual real full-column rank matrices by solving a redundant linear system involving Kronecker products. 
Additionally, Gutin \cite{gutin2021generalizations} developed two other SVD generalizations for square dual real matrices.
Other existing algorithms calculate the SVD of a dual matrix by computing the eigenvalue decomposition of its Gram matrix first \cite{pennestri2007linear}, which increases the complexity of algorithms. Recently, Qi et al. \cite{qi2022low} have revealed that a dual complex matrix can be reduced to a diagonal nonnegative dual matrix by unitary transformations. But its algorithm still needs to calculate eigenvalue decomposition first. 

Hence, we aim to reveal explicit expressions and efficient algorithms for the SVD of dual complex matrices. 
This motivates us to find necessary and sufficient conditions for the existence of the SVD first. 
We then convert the problem into an equivalent system of two matrix equations, considering the form of the SVD of dual complex matrices directly, rather than starting from the eigenvalue decomposition of Gram matrices. 
The highlight of solving the matrix equations lies in the reduced memory and time. 
For dual complex matrices failing to satisfy the SVD existence condition, we further propose the optimal low-rank approximation under a quasi-metric in the dual complex system, compared to the norm employed by Qi et al. \cite{qi2022low}.


We have found that there exist pairs of similarities between standard parts and infinitesimal parts of singular vectors of rank-2 dual real matrices numerically and proved theoretically. 
Traveling waves are identified when this property is applied to large-scale brain fMRI data,
which can explain brain activity from an innovative aspect and verify the function of the corresponding cerebral cortex.


The remainder of this paper is structured as follows. \Cref{sec:prelimi} begins by laying out some preliminary operation rules of dual complex matrices and some useful notations. \Cref{sec:CDSVD} is concerned with our main contributions to the compact dual singular value decomposition (CDSVD) of a dual complex matrix with explicit expressions and convenient algorithms. Low-rank approximation of a dual complex matrix under two different metrics is developed in \Cref{sec:low rank} and another related theoretical application of the CDSVD - the dual Moore-Penrose generalized inverse - is covered in \Cref{sec:DMPGI}. \Cref{sec:numerical} not only reveals the high efficiency of our proposed algorithm in comparison with other algorithms, but also claims the advantage of the CDSVD in extracting the true rank of the original matrix from the noisy matrix compared to the classical SVD. Traveling waves are identified in the simulation data, small-scale road monitoring video data and then large-scale brain fMRI data, which demonstrates the importance of the infinitesimal parts and validates the consistency between brain regions recognized to have streams and the corresponding cerebral cortex function. Finally, \Cref{sec:conclusions} gives a brief summary of our findings and discusses future research based on our methods.

\section{Preliminaries}\label{sec:prelimi}

In this section, we present several basic operations, properties, and notations of dual complex matrices. 

Let $\mathbb{D}$ denote the set of dual numbers. A dual number $p$ is denoted by
$p = p_{\rm{s}}+p_{\rm{i}}\epsilon\;$, where $p_{\rm{s}},p_{\rm{i}}\in\mathbb{R}$ and $\epsilon$ is the dual infinitesimal unit, satisfying $\epsilon\neq0,\; 0\epsilon=\epsilon0=0,\;1\epsilon=\epsilon1=\epsilon,\;\epsilon^2 = 0$. 
We refer respectively to $p_{\rm{s}}$ as the standard part, and $p_{\rm{i}}$ as the infinitesimal part of $p$. 
In addition, the infinitesimal unit $\epsilon$ is commutative in multiplication with complex numbers. If $p_{\rm{s}}\neq 0$, we say that $p$ is appreciable. Otherwise, $p$ is infinitesimal.

For a dual number $p$, when complex vectors $\bm{x}_{\rm{s}},\bm{x}_{\rm{i}}\in\mathbb{C}^n$ are substituted for the real numbers $p_{\rm{s}},p_{\rm{i}}$, $\bm{x} = \bm{x}_{\rm{s}}+\bm{x}_{\rm{i}}\epsilon$ is a dual complex vector and $\mathbb{DC}^{n}$ represents the set of $n$-by-1 dual complex vectors. 
Similarly, if $\bm{A}_{\rm{s}},\bm{A}_{\rm{i}}\in\mathbb{C}^{m\times n}$, then $\bm{A} = \bm{A}_{\rm{s}} + \bm{A}_{\rm{i}}\epsilon$ is a dual complex matrix and $\mathbb{DC}^{m\times n}$ stands for the set of $m$-by-$n$ dual complex matrices. The following proposition shows the preliminary properties of dual complex matrices.

\begin{proposition}\label{pro1.1}
Suppose that $\bm{A}=(a_{ij})\in\mathbb{DC}^{m\times n}$, $\bm{B}=\bm{B}_{\rm{s}} + \bm{B}_{\rm{i}}\epsilon\in\mathbb{DC}^{n\times p}$ and $\bm{C} = \bm{C}_{\rm{s}} + \bm{C}_{\rm{i}}\epsilon\in\mathbb{DC}^{n\times n}$. Then
\begin{enumerate}
    \item[${\rm{(i)}}$] The transpose of $\bm{A}$ is $\bm{A}^T = (a_{ji})$. The conjugate transpose of $\bm{A}$ is $\bm{A}^* =( \bar{a}_{ji}) $. Moreover, $(\bm{AB})^T = \bm{B}^T\bm{A}^T$, $(\bm{AB})^*=\bm{B}^*\bm{A}^*$.
    \item[${\rm{(ii)}}$] \label{ii} The dual matrix $\bm{C}$ is called diagonal if both $\bm{C}_{\rm{s}}$ and $\bm{C}_{\rm{i}}$ are diagonal matrices.
    \item[${\rm{(iii)}}$] The dual matrix $\bm{C}$ is called unitary if $\bm{C}^*\bm{C} = \bm{I}_n$, where $\bm{I}_n$ is an $n$-by-$n$ identity matrix. Moreover, $\bm{CC}^*=\bm{I}_n$.
    \item[${\rm{(iv)}}$]\label{iv} The dual matrix $\bm{C}$ is called invertible (nonsingular) if $\bm{CD} = \bm{DC} = \bm{I}_n$ for some $\bm{D}\in\mathbb{DC}^{n\times n}$. Such $\bm{D}$ is unique and denoted by $\bm{C}^{-1}$. In addition, $\bm{C}^{-1}=\bm{C}_{\rm{s}}^{-1}-\bm{C}_{\rm{s}}^{-1}\bm{C}_{\rm{i}}\bm{C}_{\rm{s}}^{-1}\epsilon $.
    \item[${\rm{(v)}}$]\label{v} Let $n\geq p$. The dual matrix $\bm{B}$ is said to have unitary columns if $\bm{B}^*\bm{B}=\bm{I}_n$. More specifically, $\bm{B}_{\rm{s}}$ is a complex matrix with unitary columns, and $\bm{B}_{\rm{s}}^*\bm{B}_{\rm{i}}+\bm{B}_{\rm{i}}^*\bm{B}_{\rm{s}}=\bm{O}_p$, where the notation $\bm{O}$ represents an all-zero matrix of appropriate size.
\end{enumerate}
\end{proposition}

\begin{proposition}[\cite{Golub1983MatrixC}]\label{pro1.2}
If $\bm{U}\in\mathbb{C}^{m\times n} (m>n)$ has unitary columns, then there exists $\widetilde{\bm{U}}\in\mathbb{C}^{m\times (m-n)}$ such that 
$\bigl[\bm{U}\;\;\widetilde{\bm{U}}\bigr]\in\mathbb{C}^{m\times m}$ is unitary.
\end{proposition}

Here, we propose new representative forms to denote dual complex matrices, which develop the equivalence of operations between matrices and dual matrices.

\begin{definition}\label{defi1.3}
Given a dual complex matrix $\bm{A} = \bm{A}_{\rm{s}}+\bm{A}_{\rm{i}}\epsilon\in\mathbb{DC}^{m\times n}$, its representative form is denoted as a block lower triangular Toeplitz matrix
\begin{align}
     \mathcal{R}(\bm{A})  := 
     \begin{bmatrix}
         \bm{A}_{\rm{s}}   &  \\
       \bm{A}_{\rm{i}}  & \bm{A}_{\rm{s}}
     \end{bmatrix}\;.
\end{align}
\end{definition}

The addition and multiplication operations of matrices are equivalent to those of dual matrices under the above representative forms. Suppose that $\bm{A},\bm{B}\in\mathbb{DC}^{m\times n}$ and $\bm{C}\in\mathbb{DC}^{n\times p}$. Then it is clear that $\mathcal{R}(\bm{A}+\bm{B})=\mathcal{R}(\bm{A})+\mathcal{R}(\bm{B})$ and $\mathcal{R}
(\bm{A}\bm{C}) = \mathcal{R}(\bm{A})\mathcal{R}(\bm{C})$. In essence, this works because the representation of $\epsilon$ has the form $\bigl[\begin{smallmatrix}
     \bm{O}   &  \\
      \bm{I}  & \bm{O}
    \end{smallmatrix}\bigr]$, which satisfies the properties of $\epsilon$. Note that for the representative form of a dual complex matrix $\bm{A}$, its conjugate transpose is $\bigl[\begin{smallmatrix}
     \bm{A}_{\rm{s}}^*   &  \\
      \bm{A}_{\rm{i}}^*  & \bm{A}_{\rm{s}}^*
\end{smallmatrix}\bigr]$ but not $\bigl[\begin{smallmatrix}
     \bm{A}_{\rm{s}}   &  \\
      \bm{A}_{\rm{i}} & \bm{A}_{\rm{s}}
\end{smallmatrix}\bigr]^*$.

For a square complex matrix $\bm{X}$, ${{\rm{sym}}}(\bm{X})$ stands for $\bm{X}+\bm{X}^*$ and ${{\rm{Diag}}}(\bm{X})$ for a matrix with only diagonal elements of $\bm{X}$. On the other side, we denote $\bm{Y}\bm{Y}^*$ by $\mathcal{P}_{\bm{Y}}$ and $\bm{I}_m-\bm{Y}\bm{Y}^*$ by $\mathcal{P}_{{\bm{Y}}^\perp}$ for an $m$-by-$n$ complex matrix $\bm{Y}$. In addition, ${\rm{diag}}(x_1,x_2,\cdots,x_n)$ stands for the diagonal matrix whose diagonal elements are $x_1, x_2, \cdots, x_n$ and ${\rm{diag}}(\bm{X}_1,\bm{X}_2,\cdots,\bm{X}_n)$ for the block diagonal matrix whose diagonal blocks are square $\bm{X}_1, \bm{X}_2, \cdots, \bm{X}_n$. Moreover, $\bm{O}_{m\times n}$ and $\bm{1}_{m\times n}$ represent an all-zero and an all-one matrices of size $m$-by-$n$, respectively. Besides, $\odot$ represents the Hadamard product (elementwise product) of two matrices with the same size. 
For the sake of use, the following sets are defined. $\mathcal{B}_{rp}:=\{\bm{A}\in\mathbb{C}^{r\times r}|\bm{A} = {\rm{diag}}(\bm{A}_{11},\bm{A}_{22},\cdots,\bm{A}_{pp}), p\leq r,\bm{A}_{ii}\in\mathbb{C}^{r_i\times r_i}\; (i=1,2,\cdots,p)\}$, $\mathcal{B}_{rp}^{\mathcal{H}}:=\{\bm{A}\in\mathcal{B}_{rp}|\bm{A} = \bm{A}^*\}$ and $\mathcal{B}_{rp}^{\mathcal{S}}:=\{\bm{A}\in\mathcal{B}_{rp}|\bm{A} = -\bm{A}^*\}$.

\section{Compact Dual SVD}\label{sec:CDSVD}

A complex matrix can be deduced to a rectangular diagonal nonnegative matrix by unitary transformations on both sides \cite{Golub1983MatrixC}. The same is true for dual matrices. Although several studies \cite{pennestri2007linear,pennestri2018moore,qi2022low} have proposed some algorithms, either the eigenvalue decomposition of the Gram matrix is used or redundant linear systems involving Kronecker products are solved. 

This section proposes the compact dual singular value decomposition (CDSVD) of dual complex matrices with explicit expressions, straightforward proofs, and convenient algorithms. 
Specifically, given a dual complex matrix $\bm{A}=\bm{A}_{\rm{s}}+\bm{A}_{\rm{i}}
\epsilon\in\mathbb{DC}^{m\times n}$, If there exist $\bm{U}=\bm{U}_{\rm{s}}+\bm{U}_{\rm{i}}\epsilon\in\mathbb{DC}^{m\times r}$,  $\bm{V}=\bm{V}_{\rm{s}}+\bm{V}_{\rm{i}}\epsilon\in\mathbb{DC}^{n\times r}$ with unitary columns and diagonal positive $\bm{\Sigma}=\bm{\Sigma}_{\rm{s}}+\bm{\Sigma}_{\rm{i}}\epsilon\in\mathbb{DR}^{r\times r}$ such that $\bm{A} = \bm{U\Sigma V}^*$, then this is called a CDSVD of $\bm{A}$, where $\bm{A}_{\rm{s}}=\bm{U}_{\rm{s}}\bm{\Sigma}_{\rm{s}}\bm{V}_{\rm{s}}^*$ is the compact SVD of $\bm{A}_{\rm{s}}$.
A dual number is positive if its standard part is positive or its standard part is equal to zero and its infinitesimal part is positive \cite{qi2022dual}. 
Thus, in the CDSVD, $\bm{\Sigma}$ is diagonal positive if and only if $\bm{\Sigma}_{\rm{i}}$ is diagonal, since $\bm{\Sigma}_{\rm{s}}$ is always diagonal positive.

The theorem below provides a necessary and sufficient condition for the existence of the CDSVD and then presents expressions of $\bm{U}_{\rm{i}}$, $\bm{V}_{\rm{i}}$, and $\bm{\Sigma}_{\rm{i}}$.


\begin{theorem}\label{the3.1summ}
Let $\bm{A}=\bm{A}_{\rm{s}}+\bm{A}_{\rm{i}}\epsilon\in\mathbb{DC}^{m\times n} (m\geq n)$. Assume $\bm{A}_{\rm{s}} = \bm{U}_{\rm{s}}\bm{\Sigma}_{\rm{s}} \bm{V}_{\rm{s}}^*$ is a compact SVD of $\bm{A}_{\rm{s}}$, where $\bm{U}_{\rm{s}}\in\mathbb{C}^{m\times r}$, $\bm{V}_{\rm{s}}\in\mathbb{C}^{n\times r}$ have unitary columns and $\bm{\Sigma}_{\rm{s}}={\rm{diag}}(\Tilde{\sigma}_1\bm{I}_{r_1},\Tilde{\sigma}_2\bm{I}_{r_2},\cdots,\Tilde{\sigma}_p\bm{I}_{r_p})\in\mathbb{R}^{r\times r}$ is diagonal positive with $\Tilde{\sigma}_1>\Tilde{\sigma}_2>\cdots>\Tilde{\sigma}_p>0$. Then the compact dual SVD (CDSVD) of $\bm{A}$ exists if and only if 
\begin{align}
    (\bm{I}_m-\bm{U}_{\rm{s}}\bm{U}_{\rm{s}}^*)\bm{A}_{\rm{i}}(\bm{I}_n-\bm{V}_{\rm{s}}\bm{V}_{\rm{s}}^*)=\bm{O}_{m\times n}\;.\label{333333}
\end{align}
Furthermore, if $\bm{A}$ has a CDSVD, then there exists a particular pair $(\bm{U}_{\rm{s}},\bm{V}_{\rm{s}})$ such that $\bm{A} = \bm{U}\bm{\Sigma}\bm{ V}^*$, where $\bm{U}=\bm{U}_{\rm{s}}+\bm{U}_{\rm{i}}\epsilon\in\mathbb{DC}^{m\times r}$ and $\bm{V}=\bm{V}_{\rm{s}}+\bm{V}_{\rm{i}}\epsilon\in\mathbb{DC}^{n\times r}$ both have unitary columns, and $\bm{\Sigma}=\bm{\Sigma}_{\rm{s}}+\bm{\Sigma}_{\rm{i}} \epsilon\in\mathbb{DR}^{r\times r}$ is diagonal positive.
Moreover, the expressions of $\bm{U}_{\rm{i}}, \bm{\Sigma}_{\rm{i}}, \bm{V}_{\rm{i}}$, the infinitesimal parts of $\bm{U},\bm{\Sigma}, \bm{V}$, are
\begin{subequations}
    \begin{align}
\bm{U}_{\rm{i}}&=\bm{U}_{\rm{s}}\left[{\rm{sym}}(\bm{R}\bm{\Sigma}_{\rm{s}})\odot\widetilde{\bm{\Delta}}+\widetilde{\bm{\Omega}}+\widetilde{\bm{\Psi}}\right]+(\bm{I}_m-\bm{U}_{\rm{s}}\bm{U}_{\rm{s}}^*)\bm{A}_{\rm{i}}\bm{V}_{\rm{s}}\bm{\Sigma}_{\rm{s}}^{-1}\;,\label{19a}\\
\bm{V}_{\rm{i}}&=\bm{V}_{\rm{s}}\Big[{\rm{sym}}(\bm{\Sigma}_{\rm{s}}\bm{R})\odot\widetilde{\bm{\Delta}}+\widetilde{\bm{\Omega}}\Big]+(\bm{I}_n-\bm{V}_{\rm{s}}\bm{V}_{\rm{s}}^*)\bm{A}_{\rm{i}}^*\bm{U}_{\rm{s}}\bm{\Sigma}_{\rm{s}}^{-1}\;,\label{19b}\\
\bm{\Sigma}_{\rm{i}} &= \frac{1}{2} {\rm{Diag}}\big({\rm{sym}}(\bm{R})\big) \;,\label{19c}
\end{align}
\end{subequations}
where $\bm{R}:=\bm{U}_{\rm{s}}^*\bm{A}_{\rm{i}}\bm{V}_{\rm{s}}$ and $\bm{R}_{ii}$ represents the $(i,i)$-block matrix of $\bm{R}$ partitioned as in $\bm{\Sigma}_{\rm{s}}$, $\bm{\widetilde{\Omega}}\in\mathcal{B}_{rp}^{\mathcal{S}}$ and 
\begin{align}
\bm{\widetilde{\Delta}} = \begin{bmatrix}
    \bm{O}_{r_1}  & \frac{\bm{1}_{r_1\times r_2}}{\Tilde{\sigma}_2^2-\Tilde{\sigma}_1^2} & \cdots & \frac{\bm{1}_{r_1\times r_p}}{\Tilde{\sigma}_p^2-\Tilde{\sigma}_1^2} \\
       \frac{\bm{1}_{r_2\times r_1}}{\Tilde{\sigma}_1^2-\Tilde{\sigma}_2^2}  & \bm{O}_{r_2} & \cdots & \frac{\bm{1}_{r_2\times r_p}}{\Tilde{\sigma}_p^2-\Tilde{\sigma}_2^2} \\
       \vdots & \vdots & \ddots & \vdots \\
       \frac{\bm{1}_{r_p\times r_1}}{\Tilde{\sigma}_1^2-\Tilde{\sigma}_p^2} & \frac{\bm{1}_{r_p\times r_2}}{\Tilde{\sigma}_2^2-\Tilde{\sigma}_p^2} & \cdots & \bm{O}_{r_p}
\end{bmatrix},
    \widetilde{\bm{\Psi}} =\begin{bmatrix}
        \frac{\bm{R}_{11}-\bm{R}_{11}^*}{2\Tilde{\sigma}_1} & &  \\
         &   \ddots & \\
         &   & \frac{\bm{R}_{pp}-\bm{R}_{pp}^*}{2\Tilde{\sigma}_p}
    \end{bmatrix}.\label{88}
\end{align}
\end{theorem}

It is worth mentioning that for a given dual complex matrix $\bm{A}$, the relationship between $\bm{A}_{\rm{i}}$ and singular vectors of $\bm{A}_{\rm{s}}$ determines whether its CDSVD exists. 
Additionally, if the dual complex matrix $\bm{A}$ satisfies \cref{333333}, then $\frac{\bm{R}_{ii}+\bm{R}_{ii}^*}{2}$ is Hermitian for any pair $(\bm{U}_{\rm{s}},\bm{V}_{\rm{s}})$. 
However, the particular pair $(\bm{U}_{\rm{s}},\bm{V}_{\rm{s}})$ ensures that the corresponding $\frac{\bm{R}_{ii}+\bm{R}_{ii}^*}{2}$ is diagonal and real in our CDSVD theorem.

In order to prove the CDSVD of dual complex matrices, we first introduce the following two lemmata.

\begin{lemma}\label{lemma3.1}
Let $\bm{\Sigma}_{\rm{s}}={\rm{diag}}(\Tilde{\sigma}_1\bm{I}_{r_1},\Tilde{\sigma}_2\bm{I}_{r_2},\cdots,\Tilde{\sigma}_p\bm{I}_{r_p})\in\mathbb{R}^{r\times r}$ with $\Tilde{\sigma}_1>\Tilde{\sigma}_2>\cdots>\Tilde{\sigma}_p>0$ and $\bm{R}\in\mathbb{C}^{r\times r}$. 
Then the matrix conditions about $\bm{P}\in\mathbb{C}^{r\times r}$ and $\bm{Q}\in\mathbb{C}^{r\times r}$ satisfying
\begin{subequations}
    \begin{align}
&\bm{P}+\bm{P}^*=\bm{O}_r\;,\label{5555555a}\\
&\bm{Q}+\bm{Q}^*=\bm{O}_r\;,\\
&\bm{\Sigma}_{R} = \bm{R}-\bm{P}\bm{\Sigma}_{\rm{s}}-\bm{\Sigma}_{\rm{s}} \bm{Q}^*\in\mathcal{B}_{rp}^{\mathcal{H}}\;,\label{666c}
    \end{align}
\end{subequations}
have solutions
\begin{subequations}
    \begin{align}
&\bm{P} = {\rm{sym}}(\bm{R}\bm{\Sigma}_{\rm{s}})\odot\widetilde{\bm{\Delta}}+\bm{\widetilde{\Omega}}+\widetilde{\bm{\Psi}}\;,\label{7a} \\
&\bm{Q} = {\rm{sym}}(\bm{\Sigma}_{\rm{s}} \bm{R})\odot\widetilde{\bm{\Delta}}+\bm{\widetilde{\Omega}}\;,\label{7b}
    \end{align}
\end{subequations}
where $\bm{\widetilde{\Omega}}\in\mathcal{B}_{rp}^{\mathcal{S}}$.
Moreover, the expressions of $\bm{\widetilde{\Delta}}$ and $\bm{\widetilde{\Psi}}$
are shown in \cref{88}.
\end{lemma}

\begin{proof}
Conjugate the conditional equation \cref{666c}, we have $\bm{\Sigma}_{R}^* = \bm{R}^*-\bm{\Sigma}_{\rm{s}} \bm{P}^*-\bm{Q}\bm{\Sigma}_{\rm{s}}$. Let $\bm{R}_{ij}, \bm{P}_{ij}, \bm{Q}_{ij}$ represent the $(i,j)$-block matrices of $\bm{R}, \bm{P}, \bm{Q}$ respectively, which are partitioned as in $\bm{\Sigma}_{\rm{s}}$. Thus, for any $i\neq j$, we can deduce the following expressions from the original matrix conditions
\begin{subequations}
    \begin{align}
&\bm{P}_{ij}+\bm{P}_{ji}^*=\bm{O}_{r_i\times r_j}\;,\label{88a}\\
&\bm{Q}_{ij}+\bm{Q}_{ji}^*=\bm{O}_{r_i\times r_j}\label{88b}\;,\\
&\bm{R}_{ij} - \Tilde{\sigma}_j\bm{P}_{ij} - \Tilde{\sigma}_i\bm{Q}_{ji}^* = \bm{O}_{r_i\times r_j}\label{88c}\;,\\
&\bm{R}_{ji}^* - \Tilde{\sigma}_i\bm{P}_{ji}^* - \Tilde{\sigma}_j\bm{Q}_{ij} = \bm{O}_{r_i\times r_j}\label{88d}\;.
    \end{align}
\end{subequations}

From $\Tilde{\sigma}_j$\cref{88c}+$\Tilde{\sigma}_i$\cref{88d}, we have $\bm{P}_{ij}=\frac{\Tilde{\sigma}_j\bm{R}_{ij}+\Tilde{\sigma}_i\bm{R}_{ji}^*}{\Tilde{\sigma}_j^2-\Tilde{\sigma}_i^2}$; from $\Tilde{\sigma}_i$\cref{88c}+$\Tilde{\sigma}_j$\cref{88d}, we have $\bm{Q}_{ij} = \frac{\Tilde{\sigma}_i\bm{R}_{ij}+\Tilde{\sigma}_j\bm{R}_{ji}^*}{\Tilde{\sigma}_j^2-\Tilde{\sigma}_i^2}$. Thereupon, the expressions for the off-diagonal blocks of $\bm{P}$ and $\bm{Q}$ are obtained. In addition, for any $i$, we have 
\begin{align*}
&\bm{R}_{ii} - \Tilde{\sigma}_i\bm{P}_{ii} - \Tilde{\sigma}_i\bm{Q}_{ii}^* = \bm{R}_{ii}^* - \Tilde{\sigma}_i\bm{P}_{ii}^* -\Tilde{\sigma}_i\bm{Q}_{ii}\;,\\
&\bm{P}_{ii}+\bm{P}_{ii}^*=\bm{O}_{r_i}\;,\\
&\bm{Q}_{ii}+\bm{Q}_{ii}^*=\bm{O}_{r_i}\;,
\end{align*}
which implies that $\bm{P}_{ii}-\bm{Q}_{ii}=\frac{\bm{R}_{ii}-\bm{R}_{ii}^*}{2\Tilde{\sigma}_i}$ and the diagonal blocks of $\bm{P}$ and $\bm{Q}$ are all skew-Hermitian matrices. 
These properties lead to the expressions of $\bm{P}$ and $\bm{Q}$ in \cref{7a,7b}.
This completes the proof.
\end{proof}

Note that when $\bm{P}$ and $\bm{Q}$ have expressions in \cref{7a,7b}, \cref{666c} becomes $\bm{\Sigma}_{R}={\rm{diag}}\big(\frac{\bm{R}_{11}+\bm{R}_{11}^*}{2},\cdots,\frac{\bm{R}_{pp}+\bm{R}_{pp}^*}{2}\big)$ and $\bm{\Sigma}_R$ is exactly block diagonal and Hermitian.

\begin{lemma}\label{le2.2}
Let $\bm{U}_{\rm{s}}\in\mathbb{C}^{m\times r}$, $\bm{V}_{\rm{s}}\in\mathbb{C}^{n\times r}$ both have unitary columns, $\bm{\Sigma}_{\rm{s}}\in\mathbb{R}^{r\times r}$ be diagonal positive, $\bm{A}_{\rm{i}}\in\mathbb{C}^{m\times n}$, $\bm{P}\in\mathbb{C}^{r\times r}$ and $\bm{Q}\in\mathbb{C}^{r\times r}$. Then the matrix equations
\begin{subequations}
    \begin{align}
        &\bm{A}_{\rm{i}} = \bm{U}_{\rm{i}}\bm{\Sigma}_{\rm{s}} \bm{V}_{\rm{s}}^*+\bm{U}_{\rm{s}}\bm{\Sigma}_{\rm{i}} \bm{V}_{\rm{s}}^*+\bm{U}_{\rm{s}}\bm{\Sigma}_{\rm{s}} \bm{V}_{\rm{i}}^*\;,\label{99a}\\
&\bm{P}=\bm{U}_{\rm{s}}^*\bm{U}_{\rm{i}}\;,\label{99b}\\
&\bm{Q}=\bm{V}_{\rm{s}}^*\bm{V}_{\rm{i}}\;,\label{99c}
    \end{align}
\end{subequations}
about $\bm{U}_{\rm{i}}\in\mathbb{C}^{m\times r}$, $\bm{V}_{\rm{i}}\in\mathbb{C}^{n\times r}$, $\bm{\Sigma}_{\rm{i}}\in\mathbb{C}^{r\times r}$ are consistent if and only if
\begin{align}
    (\bm{I}_m-\bm{U}_{\rm{s}}\bm{U}_{\rm{s}}^*)\bm{A}_{\rm{i}}(\bm{I}_n-\bm{V}_{\rm{s}}\bm{V}_{\rm{s}}^*)=\bm{O}_{m\times n}\;.\label{7}
\end{align}
Moreover, the solutions to the equations are
\begin{subequations}
    \begin{align}
&\bm{U}_{\rm{i}}=\bm{U}_{\rm{s}}\bm{P}+(\bm{I}_m-\bm{U}_{\rm{s}}\bm{U}_{\rm{s}}^*)\bm{A}_{\rm{i}}\bm{V}_{\rm{s}}\bm{\Sigma}_{\rm{s}}^{-1}\;,\label{11a}\\
&\bm{V}_{\rm{i}}=\bm{V}_{\rm{s}}\bm{Q}+(\bm{I}_n-\bm{V}_{\rm{s}}\bm{V}_{\rm{s}}^*)\bm{A}_{\rm{i}}^*\bm{U}_{\rm{s}}\bm{\Sigma}_{\rm{s}}^{-1}\;,\label{11b}\\
&\bm{\Sigma}_{\rm{i}} = \bm{U}_{\rm{s}}^*\bm{A}_{\rm{i}}\bm{V}_{\rm{s}} -  \bm{P}\bm{\Sigma}_{\rm{s}}-\bm{\Sigma}_{\rm{s}}\bm{Q}^*\label{11c}\;.
    \end{align}
\end{subequations}
\end{lemma}
\begin{proof}
Suppose that the matrix equations are satisfied. Then we obviously have $ (\bm{I}_m-\bm{U}_{\rm{s}}\bm{U}_{\rm{s}}^*)\bm{A}_{\rm{i}}(\bm{I}_n-\bm{V}_{\rm{s}}\bm{V}_{\rm{s}}^*)=\bm{O}$. When $\bm{A}_{\rm{i}}$ is pre-multiplied by $\bm{U}_{\rm{s}}^*$ and post-multiplied by $\bm{V}_{\rm{s}}$, we get $\bm{\Sigma}_{\rm{i}} = \bm{U}_{\rm{s}}^*\bm{A}_{\rm{i}}\bm{V}_{\rm{s}} -  \bm{U}_{\rm{s}}^*\bm{U}_{\rm{i}}\bm{\Sigma}_{\rm{s}}-\bm{\Sigma}_{\rm{s}}\bm{V}_{\rm{i}}^*\bm{V}_{\rm{s}}=\bm{U}_{\rm{s}}^*\bm{A}_{\rm{i}}\bm{V}_{\rm{s}} -  \bm{P}\bm{\Sigma}_{\rm{s}}-\bm{\Sigma}_{\rm{s}}\bm{Q}^*$ by applying \cref{99b,99c}. Based on \Cref{pro1.2}, there is a $\widetilde{\bm{U}}_{\rm{s}}\in\mathbb{C}^{m\times (m-r)}$ such that $\bigl[\begin{matrix}
    \bm{U}_{\rm{s}} & \widetilde{\bm{U}}_{\rm{s}}
\end{matrix}\bigr]\in\mathbb{C}^{m\times m}$ is unitary. Assume that $\bm{U}_{\rm{i}}=\bm{U}_{\rm{s}}\bm{P}+\widetilde{\bm{U}}_{\rm{s}}\widetilde{\bm{P}}$. Then premultiplication of $\bm{A}_{\rm{i}}$ by $\widetilde{\bm{U}}_{\rm{s}}^*$ results in $\widetilde{\bm{U}}_{\rm{s}}^*\bm{A}_{\rm{i}} = \widetilde{\bm{P}}\bm{\Sigma}_{\rm{s}} \bm{V}_{\rm{s}}^*$, which implies that $\widetilde{\bm{U}}_{\rm{s}}\widetilde{\bm{P}} = \widetilde{\bm{U}}_{\rm{s}} \widetilde{\bm{U}}_{\rm{s}}^*\bm{A}_{\rm{i}}\bm{V}_{\rm{s}}\bm{\Sigma}_{\rm{s}}^{-1}=(\bm{I}_m-\bm{U}_{\rm{s}}\bm{U}_{\rm{s}}^*)\bm{A}_{\rm{i}}\bm{V}_{\rm{s}}\bm{\Sigma}_{\rm{s}}^{-1}$. Hence, 
\begin{align}
    \bm{U}_{\rm{i}}=\bm{U}_{\rm{s}}\bm{P}+(\bm{I}_m-\bm{U}_{\rm{s}}\bm{U}_{\rm{s}}^*)\bm{A}_{\rm{i}}\bm{V}_{\rm{s}}\bm{\Sigma}_{\rm{s}}^{-1}\nonumber\;.
\end{align}

Similarly, there is a $\widetilde{\bm{V}}_{\rm{s}}\in\mathbb{C}^{n\times (n-r)}$ such that $\bigl[\begin{matrix}
    \bm{V}_{\rm{s}} &  \widetilde{\bm{V}}_{\rm{s}}
\end{matrix}\bigr]\in\mathbb{C}^{n\times n}$ is unitary. Assume that $\bm{V}_{\rm{i}}=\bm{V}_{\rm{s}}\bm{Q}+\widetilde{\bm{V}}_{\rm{s}}\widetilde{\bm{Q}}$. Premultiply $\bm{A}_{\rm{i}}^*$ by $\widetilde{\bm{V}}_{\rm{s}}^*$, we obtain $\widetilde{\bm{V}}_{\rm{s}}^*\bm{A}_{\rm{i}}^*=\widetilde{\bm{Q}}\bm{\Sigma}_{\rm{s}}\bm{U}_{\rm{s}}^*$, which means that $\widetilde{\bm{V}}_{\rm{s}}\widetilde{\bm{Q}} = \widetilde{\bm{V}}_{\rm{s}} \widetilde{\bm{V}}_{\rm{s}}^*\bm{A}_{\rm{i}}^*\bm{U}_{\rm{s}}\bm{\Sigma}_{\rm{s}}^{-1}=(\bm{I}_n-\bm{V}_{\rm{s}}\bm{V}_{\rm{s}}^*)\bm{A}_{\rm{i}}^*\bm{U}_{\rm{s}}\bm{\Sigma}_{\rm{s}}^{-1}$. Therefore,
\begin{align}
    \bm{V}_{\rm{i}}=\bm{V}_{\rm{s}}\bm{Q}+(\bm{I}_n-\bm{V}_{\rm{s}}\bm{V}_{\rm{s}}^*)\bm{A}_{\rm{i}}^*\bm{U}_{\rm{s}}\bm{\Sigma}_{\rm{s}}^{-1}\nonumber\;.
\end{align}

On the contrary, suppose that \cref{7} is valid and the solutions, \cref{11a,11b,11c}, are satisfied. Then we have 
\begin{align}
&\bm{U}_{\rm{i}}\bm{\Sigma}_{\rm{s}} \bm{V}_{\rm{s}}^*+\bm{U}_{\rm{s}}\bm{\Sigma}_{\rm{i}} \bm{V}_{\rm{s}}^*+\bm{U}_{\rm{s}}\bm{\Sigma}_{\rm{s}} \bm{V}_{\rm{i}}^*\nonumber\\
    =&\bm{U}_{\rm{s}}\bm{P}\bm{\Sigma}_{\rm{s}} \bm{V}_{\rm{s}}^* + \mathcal{P}_{(\bm{U}_{\rm{s}})^{\perp}}\bm{A}_{\rm{i}}\bm{V}_{\rm{s}}\bm{V}_{\rm{s}}^* +\bm{U}_{\rm{s}}\bm{\Sigma}_{\rm{i}} \bm{V}_{\rm{s}}^*+\bm{U}_{\rm{s}}\bm{\Sigma}_{\rm{s}}\bm{Q}^*\bm{V}_{\rm{s}}^* + \bm{U}_{\rm{s}}\bm{U}_{\rm{s}}^*\bm{A}_{\rm{i}}\mathcal{P}_{(\bm{V}_{\rm{s}})^{\perp}}\nonumber\\
    =& \bm{A}_{\rm{i}}-(\bm{I}_m-\bm{U}_{\rm{s}}\bm{U}_{\rm{s}}^*)\bm{A}_{\rm{i}}(\bm{I}_n-\bm{V}_{\rm{s}}\bm{V}_{\rm{s}}^*)+\bm{U}_{\rm{s}}\left(\bm{P}\bm{\Sigma}_{\rm{s}}+\bm{\Sigma}_{\rm{i}}+\bm{\Sigma}_{\rm{s}}\bm{Q}^*-\bm{U}_{\rm{s}}^*\bm{A}_{\rm{i}}\bm{V}_{\rm{s}}\right)\bm{V}_{\rm{s}}^*\nonumber\\
    =& \bm{A}_{\rm{i}}\nonumber\;.
\end{align}

Moreover, $\bm{U}_{\rm{s}}^*\bm{U}_{\rm{i}} = \bm{P}$ and $\bm{V}_{\rm{s}}^*\bm{V}_{\rm{i}} = \bm{Q}$. Accordingly, the matrix equations hold. Here, we complete the proof.
\end{proof}



From the two aforementioned lemmata, 
we induce the proof of \Cref{the3.1summ}.

\begin{proof}
Suppose that $\bm{A}$ has a CDSVD with the following representative form
\begin{align}\label{4}
     \left[\begin{array}{cc}
      \bm{A}_{\rm{s}}   &  \\
      \bm{A}_{\rm{i}}  & \bm{A}_{\rm{s}}
    \end{array}\right] = \left[\begin{array}{cc}
      \bm{U}_{\rm{s}}   &  \\
      \bm{U}_{\rm{i}}  & \bm{U}_{\rm{s}}
    \end{array}\right]\left[\begin{array}{cc}
      \bm{\Sigma}_{\rm{s}}   &  \\
      \bm{\Sigma}_{\rm{i}}  & \bm{\Sigma}_{\rm{s}}
    \end{array}\right]\left[\begin{array}{cc}
      \bm{V}^*_{\rm{s}}   &  \\
      \bm{V}^*_{\rm{i}}  & \bm{V}^*_{\rm{s}}
    \end{array}\right]\;.
\end{align}
Then we have $\bm{A}_{\rm{s}}=\bm{U}_{\rm{s}}\bm{\Sigma}_{\rm{s}}\bm{V}_{\rm{s}}^*$ and $\bm{A}_{\rm{i}} = \bm{U}_{\rm{i}}\bm{\Sigma}_{\rm{s}} \bm{V}_{\rm{s}}^*+\bm{U}_{\rm{s}}\bm{\Sigma}_{\rm{i}} \bm{V}_{\rm{s}}^*+\bm{U}_{\rm{s}}\bm{\Sigma}_{\rm{s}} \bm{V}_{\rm{i}}^*$, which implies that \cref{333333} is satisfied. 

Assume that $\bm{A}_{\rm{s}}=\bm{M}_{\rm{s}}\bm{\Sigma}_{\rm{s}}\bm{N}_{\rm{s}}^*$ is a compact SVD of $\bm{A}_{\rm{s}}$. Let $\widehat{\bm{R}}:=\bm{M}_{\rm{s}}^*\bm{A}_{\rm{i}}\bm{N}_{\rm{s}}$. Then $\bm{\Sigma}_{\widehat{R}} = \widehat{\bm{R}}-\bm{\widehat{P}\Sigma}_{\rm{s}}-\bm{\Sigma}_{\rm{s}}\widehat{\bm{Q}}^*$ is block diagonal and Hermitian when $\widehat{\bm{P}}$ and $\widehat{\bm{Q}}$ is taken as in \cref{7a,7b} in accordance with \Cref{lemma3.1}. 
Specifically, $\bm{\Sigma}_{\widehat{R}}={\rm{diag}}\big(\frac{\widehat{\bm{R}}_{11}+\widehat{\bm{R}}_{11}^*}{2},\cdots,\frac{\widehat{\bm{R}}_{pp}+\widehat{\bm{R}}_{pp}^*}{2}\big)$. 
Thus, for any $t=1,2,\cdots,p$, there is a unitary matrix $\bm{X}_{tt}\in\mathbb{C}^{r_t\times r_t}$ and a diagonal real matrix $\bm{\Lambda}_{tt}\in\mathbb{R}^{r_t\times r_t}$ such that $\frac{\bm{\widehat{R}}_{tt}+\bm{\widehat{R}}_{tt}^*}{2} = \bm{X}_{tt}\bm{\Lambda}_{tt}\bm{X}_{tt}^*$, where $\bm{\Lambda}_{tt} = {\rm{diag}}(\sigma_{t,1},\sigma_{t,2},\cdots,\sigma_{t,r_t})$ satisfying $\sigma_{t,1}\geq\sigma_{t,2}\geq\cdots\geq\sigma_{t,r_t}$. Denote $\bm{X}:={\rm{diag}}(\bm{X}_{11},\cdots,\bm{X}_{pp})$. 

Set $\bm{U}_{\rm{s}}=\bm{M}_{\rm{s}}\bm{X}$, $\bm{V}_{\rm{s}}=\bm{N}_{\rm{s}}\bm{X}$ and $\bm{R}=\bm{U}_{\rm{s}}^*\bm{A}_{\rm{i}}\bm{V}_{\rm{s}}$. Next, we prove that this pair $(\bm{U}_{\rm{s}},\bm{V}_{\rm{s}})$ is the particular pair that can ensure the corresponding $\frac{\bm{R}_{tt}+\bm{R}_{tt}^*}{2}$ to be diagonal and real.
On the one hand, $\bm{U}_{\rm{s}}\bm{\Sigma}_{\rm{s}}\bm{V}_{\rm{s}}^* = \bm{M}_{\rm{s}}\bm{X}\bm{\Sigma}_{\rm{s}}\bm{X}^*\bm{N}_{\rm{s}}^* = \bm{M}_{\rm{s}}\bm{\Sigma}_{\rm{s}}\bm{N}_{\rm{s}}^* = \bm{A}_{\rm{s}}$, which means that $\bm{U}_{\rm{s}}\bm{\Sigma}_{\rm{s}}\bm{V}_{\rm{s}}^*$ is also a compact SVD of $\bm{A}_{\rm{s}}$.
On the other hand, note that $\bm{R} = \bm{U}_{\rm{s}}^*\bm{A}_{\rm{i}}\bm{V}_{\rm{s}} = \bm{X}^*\bm{M}_{\rm{s}}^*\bm{A}_{\rm{i}}\bm{N}_{\rm{s}}\bm{X} = \bm{X}^*\bm{\widehat{R}}\bm{X}$, then we have $\frac{\bm{R}_{tt}+\bm{R}_{tt}^*}{2} = \bm{X}_{tt}^*\frac{\bm{\widehat{R}}_{tt}+\bm{\widehat{R}}_{tt}^*}{2}\bm{X}_{tt}=\bm{\Lambda}_{tt}$. 

By using this particular pair $(\bm{U}_{\rm{s}},\bm{V}_{\rm{s}})$, we then aim to solve $(\bm{U}_{\rm{i}},\bm{V}_{\rm{i}},\bm{\Sigma}_{\rm{i}})$ such that \cref{4} holds.
Denoting $\bm{P}:= \bm{U}_{\rm{s}}^*\bm{U}_{\rm{i}}$ and $\bm{Q}:= \bm{V}_{\rm{s}}^*\bm{V}_{\rm{i}}$, we get $\bm{\Sigma}_{\rm{i}} = \bm{R}-\bm{P}\bm{\Sigma}_{\rm{s}}-\bm{\Sigma}_{\rm{s}} \bm{Q}^*$ when $\bm{A}_{\rm{i}}$ is pre-multiplied by $\bm{U}_{\rm{s}}^*$ and post-multiplied by $\bm{V}_{\rm{s}}$. 
Moreover, since $\bm{U}, \bm{V}$ both have unitary columns, $\bm{P}+\bm{P}^*=\bm{O}$ and  $\bm{Q}+\bm{Q}^*=\bm{O}$ hold by applying \Cref{pro1.1} $({\rm{v}})$. 
Hence, from \Cref{lemma3.1}, when $\bm{P}$ and $\bm{Q}$ are expressed as \cref{7a,7b}, we have $\bm{R}-\bm{P}\bm{\Sigma}_{\rm{s}}-\bm{\Sigma}_{\rm{s}} \bm{Q}^* = {\rm{diag}}\big(\frac{\bm{R}_{11}+\bm{R}_{11}^*}{2},\cdots,\frac{\bm{R}_{pp}+\bm{R}_{pp}^*}{2}\big) = {\rm{diag}}(\bm{\Lambda}_{11},\cdots,\bm{\Lambda}_{pp})$. That is, $\bm{\Sigma}_{\rm{i}} = {\rm{diag}}(\sigma_{1,1},\cdots,\sigma_{1,r_1},\cdots,\sigma_{p,1},\cdots,\sigma_{p,r_p})$ is diagonal and real.

It then yields the forms of $\bm{U}_{\rm{i}}$ and $\bm{V}_{\rm{i}}$ in \cref{11a,11b} based on \Cref{le2.2}. 
Since $\bm{\Sigma}_{\rm{s}}$, $\bm{\Sigma}_{\rm{i}}$ are diagonal real and $\bm{P}$, $\bm{Q}$ are skew-Hermitian, we obtain $\bm{\Sigma}_{\rm{i}}=\frac{1}{2}(\bm{\Sigma}_{\rm{i}}+\bm{\Sigma}_{\rm{i}}^*)=\frac{1}{2}{\rm{Diag}}\big(\bm{R}+\bm{R}^*+(\bm{\Sigma}_{\rm{s}}\bm{P}-\bm{P}\bm{\Sigma}_{\rm{s}})+(\bm{\Sigma}_{\rm{s}}\bm{Q}-\bm{Q}\bm{\Sigma}_{\rm{s}})\big)=\frac{1}{2}{\rm{Diag}}(\bm{R}+\bm{R}^*)$.

Conversely, suppose that \eqref{333333} is valid with the solutions \cref{19a,19b,19c}, then from \Cref{le2.2}, we know that the matrix equation \eqref{99a} is satisfied. This implies that the representative form \eqref{4} of the CDSVD of $\bm{A}$ is also satisfied.

Consequently, we complete the proof.
\end{proof}

In summary, if the dual complex matrix $\bm{A}$ has a CDSVD, then its $r$ singular values fulfill that $\Tilde{\sigma}_1+\sigma_{1,1}\epsilon\geq\cdots\geq\Tilde{\sigma}_1+\sigma_{1,r_1}\epsilon\geq\cdots\geq\Tilde{\sigma}_p+\sigma_{p,1}\epsilon\geq\cdots\geq\Tilde{\sigma}_p+\sigma_{p,r_p}\epsilon>0$.
\Cref{alg:CDSVD} shows the method of calculating the CDSVD of a given dual complex matrix $\bm{A}\in\mathbb{DC}^{m\times n} (m\geq n)$, where the free matrix $\widetilde{\bm{\Omega}}$ is taken to be the zero matrix.

\begin{algorithm}[ht]
\caption{The Compact Dual SVD (CDSVD) of $\bm{A}=\bm{U\Sigma V}^*\in\mathbb{DC}^{m\times n}$.}
\label{alg:CDSVD}
\begin{algorithmic}
\REQUIRE
{$\bm{A}=\bm{A}_{\rm{s}}+\bm{A}_{\rm{i}}\epsilon\in\mathbb{DC}^{m\times n}\; (m\geq n)$.}
\STATE{\textbf{Step 1}. Decompose $\bm{A}_{\rm{s}}$ by a compact SVD $\bm{A}_{\rm{s}} = \bm{U}_{\rm{s}}\bm{\Sigma}_{\rm{s}}\bm{V}_{\rm{s}}^*$, where $\bm{U}_{\rm{s}}\in\mathbb{C}^{m\times r}, \bm{V}_{\rm{s}}\in\mathbb{C}^{n\times r}$ both have unitary columns and $\bm{\Sigma}_{\rm{s}}={\rm{diag}}(\Tilde{\sigma}_1\bm{I}_{r_1},\Tilde{\sigma}_2\bm{I}_{r_2},\cdots,\Tilde{\sigma}_p\bm{I}_{r_p})\in\mathbb{R}^{r\times r}$ is diagonal positive satisfying $\Tilde{\sigma}_1>\Tilde{\sigma}_2>\cdots>\Tilde{\sigma}_p>0$.}
\WHILE{the condition  $(\bm{I}_m-\bm{U}_{\rm{s}}\bm{U}_{\rm{s}}^*)\bm{A}_{\rm{i}}(\bm{I}_n-\bm{V}_{\rm{s}}\bm{V}_{\rm{s}}^*)=\bm{O}_{m\times n}$ is satisfied}
\STATE{\textbf{Step 2}.
    Compute $\bm{R} = \bm{U}_{\rm{s}}^*\bm{A}_{\rm{i}}\bm{V}_{\rm{s}}=\left[\begin{array}{cccc}
        \bm{R}_{11} & \bm{R}_{12} & \cdots & \bm{R}_{1p} \\
        \bm{R}_{21} & \bm{R}_{22} & \cdots & \bm{R}_{2p} \\
        \vdots & \vdots & \ddots & \vdots \\
        \bm{R}_{p1} & \bm{R}_{p2} & \cdots & \bm{R}_{pp} \\
    \end{array}\right]$.}
\STATE{\textbf{Step 3}.
Compute the eigenvalue decomposition of $\frac{\bm{R}_{tt}+\bm{R}_{tt}^*}{2}$ for $t=1,2,\cdots,p$.\\
\quad $\frac{\bm{R}_{tt}+\bm{R}_{tt}^*}{2}= \bm{X}_{tt}\bm{\Lambda}_{tt}\bm{X}_{tt}^*$, where $\bm{X}_{tt}\in\mathbb{C}^{r_t\times r_t}$ is unitary and $\bm{\Lambda}_{tt}\; = \;{\rm{diag}}\left(\sigma_{t,1},\right.$\\
\quad $\left.\sigma_{t,2},\;\cdots,\;\sigma_{t,r_t}\right)\in\mathbb{R}^{r_t\times r_t}$ is diagonal real satisfying $\sigma_{t,1}\geq\sigma_{t,2}\geq\cdots\geq\sigma_{t,r_t}$.\\
\quad Set $\bm{X} = {\rm{diag}}(\bm{X}_{11},\bm{X}_{22},\cdots,\bm{X}_{pp})$.}
\STATE{\textbf{Step 4.}
Update $\bm{U}_{\rm{s}}\leftarrow \bm{U}_{\rm{s}}\bm{X}$, $\bm{V}_{\rm{s}}\leftarrow \bm{V}_{\rm{s}}\bm{X}$ and $\bm{R}\leftarrow \bm{X}^*\bm{R}\bm{X}$.
}
\STATE{\textbf{Step 5}.
    Compute $\bm{P} = 
    \left[\begin{array}{cccc}
       \frac{\bm{R}_{11}-\bm{R}_{11}^*}{2\Tilde{\sigma}_1}   & \frac{\Tilde{\sigma}_2\bm{R}_{12}+\Tilde{\sigma}_1\bm{R}_{21}^*}{\Tilde{\sigma}_2^2-\Tilde{\sigma}_1^2} & \cdots & \frac{\Tilde{\sigma}_p\bm{R}_{1p}+\Tilde{\sigma}_1\bm{R}_{p1}^*}{\Tilde{\sigma}_p^2-\Tilde{\sigma}_1^2} \\
       \frac{\Tilde{\sigma}_1\bm{R}_{21}+\Tilde{\sigma}_2\bm{R}_{12}^*}{\Tilde{\sigma}_1^2-\Tilde{\sigma}_2^2}  & \frac{\bm{R}_{22}-\bm{R}_{22}^*}{2\Tilde{\sigma}_2} & \cdots & \frac{\Tilde{\sigma}_p\bm{R}_{2p}+\Tilde{\sigma}_2\bm{R}_{p2}^*}{\Tilde{\sigma}_p^2-\Tilde{\sigma}_2^2} \\
       \vdots & \vdots & \ddots & \vdots \\
       \frac{\Tilde{\sigma}_1\bm{R}_{p1}+\Tilde{\sigma}_p\bm{R}_{1p}^*}{\Tilde{\sigma}_1^2-\Tilde{\sigma}_p^2} & \frac{\Tilde{\sigma}_2\bm{R}_{p2}+\Tilde{\sigma}_p\bm{R}_{2p}^*}{\Tilde{\sigma}_2^2-\Tilde{\sigma}_p^2} & \cdots & \frac{\bm{R}_{pp}-\bm{R}_{pp}^*}{2\Tilde{\sigma}_p}
    \end{array}\right] $\;.}
\STATE{\textbf{Step 6}.
    Compute $\bm{Q} = 
    \left[\begin{array}{cccc}
       \bm{O}_{r_1}   & \frac{\Tilde{\sigma}_1\bm{R}_{12}+\Tilde{\sigma}_2\bm{R}_{21}^*}{\Tilde{\sigma}_2^2-\Tilde{\sigma}_1^2} & \cdots & \frac{\Tilde{\sigma}_1\bm{R}_{1p}+\Tilde{\sigma}_p\bm{R}_{p1}^*}{\Tilde{\sigma}_p^2-\Tilde{\sigma}_1^2} \\
       \frac{\Tilde{\sigma}_2\bm{R}_{21}+\Tilde{\sigma}_1\bm{R}_{12}^*}{\Tilde{\sigma}_1^2-\Tilde{\sigma}_2^2}  & \bm{O}_{r_2} & \cdots & \frac{\Tilde{\sigma}_2\bm{R}_{2p}+\Tilde{\sigma}_p\bm{R}_{p2}^*}{\Tilde{\sigma}_p^2-\Tilde{\sigma}_2^2} \\
       \vdots & \vdots & \ddots & \vdots \\
       \frac{\Tilde{\sigma}_p\bm{R}_{p1}+\Tilde{\sigma}_1\bm{R}_{1p}^*}{\Tilde{\sigma}_1^2-\Tilde{\sigma}_p^2} & \frac{\Tilde{\sigma}_p\bm{R}_{p2}+\Tilde{\sigma}_2\bm{R}_{2p}^*}{\Tilde{\sigma}_2^2-\Tilde{\sigma}_p^2} & \cdots & \bm{O}_{r_p}
    \end{array}\right] $\;.}
\STATE{\textbf{Step 7}. Set $\bm{U}_{\rm{i}}=\bm{U}_{\rm{s}}\bm{P}+(\bm{I}_m-\bm{U}_{\rm{s}}\bm{U}_{\rm{s}}^*)\bm{A}_{\rm{i}}\bm{V}_{\rm{s}}\bm{\Sigma}_{\rm{s}}^{-1}$\;;}
\STATE{\quad\quad\quad\quad\quad\; $\bm{V}_{\rm{i}}  =\bm{V}_{\rm{s}}\bm{Q}+(\bm{I}_n-\bm{V}_{\rm{s}}\bm{V}_{\rm{s}}^*)\bm{A}_{\rm{i}}^*\bm{U}_{\rm{s}}\bm{\Sigma}_{\rm{s}}^{-1}$\;;}
\STATE{\quad\quad\quad\quad\quad\; $\bm{\Sigma}_{\rm{i}}  = \frac{1}{2} {\rm{Diag}}(\bm{R}+\bm{R}^*)$\;.}
\ENDWHILE
\ENSURE{Dual complex matrices $\bm{U} = \bm{U}_{\rm{s}}+\bm{U}_{\rm{i}}\epsilon\in\mathbb{DC}^{m\times r}$ and $\bm{V} = \bm{V}_{\rm{s}}+\bm{V}_{\rm{i}}\epsilon\in\mathbb{DC}^{n\times r}$ both having unitary columns; diagonal positive dual matrix $\bm{\Sigma} = \bm{\Sigma}_{\rm{s}}+\bm{\Sigma}_{\rm{i}}\epsilon\in\mathbb{DR}^{r\times r}$.}
\end{algorithmic}
\end{algorithm}

In most practical and numerical problems, the nonzero singular values of matrices can be regarded as single, i.e., with multiplicity one, which can deduce as a simple case for the CDSVD. 

\begin{theorem}\label{the2.1}
Let $\bm{A}=\bm{A}_{\rm{s}}+\bm{A}_{\rm{i}}\epsilon\in\mathbb{DC}^{m\times n} (m\geq n)$. Assume $\bm{A}_{\rm{s}} = \bm{U}_{\rm{s}}\bm{\Sigma}_{\rm{s}} \bm{V}_{\rm{s}}^*$ is a compact SVD of $\bm{A}_{\rm{s}}$, where $\bm{U}_{\rm{s}}\in\mathbb{C}^{m\times r}$, $\bm{V}_{\rm{s}}\in\mathbb{C}^{n\times r}$ have unitary columns and $\bm{\Sigma}_{\rm{s}}={\rm{diag}}(\sigma_1,\sigma_2,\cdots,\sigma_r)\in\mathbb{R}^{r\times r}$ is diagonal positive with $\sigma_1>\sigma_2>\cdots>\sigma_r>0$. Then the compact dual SVD (CDSVD) of $\bm{A}$ exists if and only if 
\begin{align}
    (\bm{I}_m-\bm{U}_{\rm{s}}\bm{U}_{\rm{s}}^*)\bm{A}_{\rm{i}}(\bm{I}_n-\bm{V}_{\rm{s}}\bm{V}_{\rm{s}}^*)=\bm{O}_{m\times n}\;.\label{12}
\end{align}
Furthermore, if $\bm{A}$ has a CDSVD $\bm{A} = \bm{U}\bm{\Sigma}\bm{ V}^*$, where $\bm{U}=\bm{U}_{\rm{s}}+\bm{U}_{\rm{i}}\epsilon\in\mathbb{DC}^{m\times r}$ and $\bm{V}=\bm{V}_{\rm{s}}+\bm{V}_{\rm{i}}\epsilon\in\mathbb{DC}^{n\times r}$ both have unitary columns, and $\bm{\Sigma}=\bm{\Sigma}_{\rm{s}}+\bm{\Sigma}_{\rm{i}} \epsilon\in\mathbb{DR}^{r\times r}$ is diagonal positive, then
\begin{subequations}
    \begin{align}
&\bm{U}_{\rm{i}}=\bm{U}_{\rm{s}}\left[{\rm{sym}}(\bm{R}\bm{\Sigma}_{\rm{s}})\odot\bm{\Delta}+\bm{\Omega}+\frac{{\rm{Diag}}(\bm{R}-\bm{R}^*)}{2\bm{\Sigma}_{\rm{s}}}\right]+(\bm{I}_m-\bm{U}_{\rm{s}}\bm{U}_{\rm{s}}^*)\bm{A}_{\rm{i}}\bm{V}_{\rm{s}}\bm{\Sigma}_{\rm{s}}^{-1},\label{100a}\\
&\bm{V}_{\rm{i}}=\bm{V}_{\rm{s}}\big[{\rm{sym}}(\bm{\Sigma}_{\rm{s}} \bm{R})\odot\bm{\Delta}+\bm{\Omega}\big]+(\bm{I}_n-\bm{V}_{\rm{s}}\bm{V}_{\rm{s}}^*)\bm{A}_{\rm{i}}^*\bm{U}_{\rm{s}}\bm{\Sigma}_{\rm{s}}^{-1}\;,\label{100b}\\
&\bm{\Sigma}_{\rm{i}} = \frac{1}{2} {\rm{Diag}}\big({\rm{sym}}(\bm{R})\big) \;,\label{100c}
    \end{align}
\end{subequations}
where $\bm{R}:=\bm{U}_{\rm{s}}^*\bm{A}_{\rm{i}}\bm{V}_{\rm{s}}$, the division is an elementwise operation for only diagonal elements, $\bm{\Omega}\in\mathcal{B}_{rr}^{\mathcal{S}}$
and
\begin{align}\label{6}
    \bm{\Delta} = \left[\begin{array}{cccc}
       0  & \frac{1}{\sigma_2^2-\sigma_1^2} & \cdots & \frac{1}{\sigma_r^2-\sigma_1^2} \\
       \frac{1}{\sigma_1^2-\sigma_2^2}  & 0 & \cdots & \frac{1}{\sigma_r^2-\sigma_2^2} \\
       \vdots & \vdots & \ddots & \vdots \\
       \frac{1}{\sigma_1^2-\sigma_r^2} & \frac{1}{\sigma_2^2-\sigma_r^2} & \cdots & 0
    \end{array}\right].
\end{align}
\end{theorem}

While \Cref{the3.1summ} provides the expressions of $\bm{U}_{\rm{i}}$ and $\bm{V}_{\rm{i}}$, they are not unique since $\bm{\widetilde{\Omega}}$ has some degrees of freedom. 
Especially, when considering the common case shown in \Cref{the2.1}, $\bm{\Omega}\in\mathcal{B}_{rr}^{\mathcal{S}}$ implies that it is an arbitrary diagonal pure imaginary matrix.
Actually, the value of $\bm{\Omega}$ can be fixed if some elements of $\bm{V}$ are required to be dual numbers. The following corollary summarizes this.

\begin{corollary}
Let $\bm{A}=\bm{A}_{\rm{s}}+\bm{A}_{\rm{i}}\epsilon=\bm{U\Sigma V^*}\in\mathbb{DC}^{m\times n} (m\geq n)$ be the CDSVD of $\bm{A}$ and the nonzero singular values of $\bm{A}_{s}$ is with multiplicity one. For the $t$-th column of $\bm{V}_{\rm{s}}$, there exists a nonzero real element, which is represented by $\bm{V}_{\rm{s}}(\ell_t,t)$. Then the corresponding $(\ell_t,t)$ element of $\bm{V}_{\rm{i}}$ is also real if and only if the expression of $\bm{\Omega}$ is
\begin{align}
    \bm{\Omega} &= {\rm{diag}}(\omega_1\imath,\omega_2\imath,\cdots,\omega_r\imath)\nonumber\;,\\
    \omega_t  &= \frac{-{\rm{Im}}(\widetilde{\bm{V}}_{\rm{i}}(\ell_t,t))}{\bm{V}_{\rm{s}}(\ell_t,t)}\;,\;(t=1,2,\cdots,r)\;,\nonumber
\end{align}
where $\widetilde{\bm{V}}_{\rm{i}}=\bm{V}_{\rm{s}}\big[{\rm{sym}}(\bm{\Sigma}_{\rm{s}} \bm{R})\odot\bm{\Delta}\big]+(\bm{I}_n-\bm{V}_{\rm{s}}\bm{V}_{\rm{s}}^*)\bm{A}_{\rm{i}}^*\bm{U}_{\rm{s}}\bm{\Sigma}_{\rm{s}}^{-1}$ and ${\rm{Im}}(v)$ represents the imaginary part of the complex number $v$.
\end{corollary}

\section{Low Rank Approximation of Dual Complex Matrices}\label{sec:low rank}

In theory, the sum of the first $k$ components of the SVD provides the optimal approximation for any matrices with rank at most $k$ \cite{stewart1993early}. 
In practice, a large-scale data matrix can be compressed into a low-rank matrix that still contains the essence of the information. Similarly, for dual complex matrices, low rank approximation is also one of the most important applications of the CDSVD. We begin with the definition of the rank-$r$ dual complex matrices.

\subsection{Rank of Dual Complex Matrices}
It is evident that a rank-$r$ $m$-by-$n$ complex matrix can be represented as $\bm{U\Sigma V}^*$, where $\bm{U}\in\mathbb{C}^{m\times r}$, $\bm{V}\in\mathbb{C}^{n\times r}$ have unitary columns and $\bm{\Sigma}\in\mathbb{R}^{r\times r}$ is diagonal positive. 
We wonder if this holds for dual complex matrices as well.

Wang et al. \cite{wang2022dual} defined a full column rank dual matrix, whose standard part is a full column rank matrix.
They then developed the dual rank-$r$ decomposition. 
An $m$-by-$n$ dual complex matrix $\bm{A}$ whose standard part $\bm{A}_{\rm{s}}$ is of rank $r$ can be decomposed as the product of two $r$ full rank dual matrices, $\bm{A} = \bm{BC}$, if and only if the $r$ full rank decomposition $\bm{A}_{\rm{s}} =\bm{B}_{\rm{s}}\bm{C}_{\rm{s}}$ satisfies $\left(\bm{I}_m - \bm{B}_{\rm{s}}\bm{B}_{\rm{s}}^\dag\right)\bm{A}_{\rm{i}}\left(\bm{I}_n-\bm{C}_{\rm{s}}^\dag\bm{C}_{\rm{s}}\right)=\bm{O}_{m \times n}$.

Inspired by the above dual rank-$r$ decomposition, we claim that its existence and that of the CDSVD are equivalent in the following theorem.
\begin{theorem}\label{equiva}
Let $\bm{A}=\bm{A}_{\rm{s}}+\bm{A}_{\rm{i}}\epsilon\in\mathbb{DC}^{m\times n} (m\geq n)$. Assume that $\bm{A}_{\rm{s}} = \bm{U}_{\rm{s}}\bm{\Sigma}_{\rm{s}} \bm{V}_{\rm{s}}^*$ is a compact SVD of $\bm{A}_{\rm{s}}$, where $\bm{U}_{\rm{s}}\in\mathbb{C}^{m\times r}$, $\bm{V}_{\rm{s}}\in\mathbb{C}^{n\times r}$ both have unitary columns, and $\bm{\Sigma}_{\rm{s}}\in\mathbb{R}^{r\times r}$ is diagonal positive. Then the following conditions are equivalent:
\begin{enumerate}[(i)]
    \item[${\rm{(i)}}$] $(\bm{I}_m-\bm{U}_{\rm{s}}\bm{U}_{\rm{s}}^*)\bm{A}_{\rm{i}}(\bm{I}_n-\bm{V}_{\rm{s}}\bm{V}_{\rm{s}}^*)=\bm{O}_{m\times n}$;
    \item[${\rm{(ii)}}$] The CDSVD of $\bm{A}$ exists;
    \item[${\rm{(iii)}}$] The dual rank-$r$ decomposition of $\bm{A}$ exists.
\end{enumerate}
\end{theorem}
\begin{proof}
According to \Cref{the3.1summ}, we get (i)$\Leftrightarrow$(ii).

(ii)$\Rightarrow$(iii): From \Cref{the3.1summ}, suppose that the CDSVD of $\bm{A}$ exists and there is a particular pair $(\bm{U}_{\rm{s}},\bm{V}_{\rm{s}})$ such that $\bm{A} = \bm{U\Sigma V}^*$,  where $\bm{U}=\bm{U}_{\rm{s}}+\bm{U}_{\rm{i}}\epsilon\in\mathbb{DC}^{m\times r}$ and $\bm{V}=\bm{V}_{\rm{s}}+\bm{V}_{\rm{i}}\epsilon\in\mathbb{DC}^{n\times r}$ both have unitary columns, and $\bm{\Sigma}=\bm{\Sigma}_{\rm{s}}+\bm{\Sigma}_{\rm{i}} \epsilon\in\mathbb{DR}^{r\times r}$ is diagonal positive. 
Let $\bm{B} = \bm{U}$ and $\bm{C} = \bm{\Sigma V}^*$, then $\bm{B}_{\rm{s}} = \bm{U}_{\rm{s}}$ is of $r$ full column rank and $\bm{C}_{\rm{s}} = \bm{\Sigma}_{\rm{s}}\bm{V}_{\rm{s}}^*$ is of $r$ full row rank. 
Moreover, $\left(\bm{I}_m - \bm{B}_{\rm{s}}\bm{B}_{\rm{s}}^\dag\right)\bm{A}_{\rm{i}}\left(\bm{I}_n-\bm{C}_{\rm{s}}^\dag\bm{C}_{\rm{s}}\right)=(\bm{I}_m-\bm{U}_{\rm{s}}\bm{U}_{\rm{s}}^*)\bm{A}_{\rm{i}}(\bm{I}_n-\bm{V}_{\rm{s}}\bm{V}_{\rm{s}}^*)=\bm{O}_{m\times n}$. Hence, $\bm{A}=\bm{BC}$ is a dual rank-$r$ decomposition of $\bm{A}$.

(iii)$\Rightarrow$(i): Suppose that the dual rank-$r$ decomposition of $\bm{A}$ exists and has the form $\bm{A}=\bm{BC}$, where $\bm{B}= \bm{B}_{\rm{s}}+\bm{B}_{\rm{i}}\epsilon\in\mathbb{DC}^{m\times r}$ is of $r$ full column rank and $\bm{C}= \bm{C}_{\rm{s}}+\bm{C}_{\rm{i}}\epsilon\in\mathbb{DC}^{r\times n}$ is of $r$ full row rank. 
Then we have $\bm{A}_{\rm{s}}=\bm{B}_{\rm{s}}\bm{C}_{\rm{s}}$.
On account of the fact that $\bm{A}_{\rm{s}} = \bm{U}_{\rm{s}}\bm{\Sigma}_{\rm{s}} \bm{V}_{\rm{s}}^*$ is a compact SVD of $\bm{A}_{\rm{s}}$, column spaces fulfill ${\rm{Ran}}(\bm{U}_{\rm{s}})={\rm{Ran}}(\bm{B}_{\rm{s}})$ and ${\rm{Ran}}(\bm{V}_{\rm{s}})={\rm{Ran}}(\bm{C}^*_{\rm{s}})$.
Thus, there exist nonsingular $\bm{X}\in\mathbb{C}^{r\times r}$ and $\bm{Y}\in\mathbb{C}^{r\times r}$ such that $\bm{B}_{\rm{s}} = \bm{U}_{\rm{s}}\bm{X}$ and $\bm{C}_{\rm{s}}^* = \bm{V}_{\rm{s}}\bm{Y}$. It can be
inferred that $\bm{B}_{\rm{s}}^\dag = \bm{X}^{-1}\bm{U}_{\rm{s}}^*$ and $\bm{C}_{\rm{s}}^\dag = \bm{V}_{\rm{s}}(\bm{Y}^*)^{-1}$. 
Moreover, based on the existence condition of the dual rank-$r$ decomposition, we have $\left(\bm{I}_m - \bm{B}_{\rm{s}}\bm{B}_{\rm{s}}^\dag\right)\bm{A}_{\rm{i}}\left(\bm{I}-\bm{C}_{\rm{s}}^\dag\bm{C}_{\rm{s}}\right)=\bm{O}_{m\times n}$, that is, $(\bm{I}_m-\bm{U}_{\rm{s}}\bm{U}_{\rm{s}}^*)\bm{A}_{\rm{i}}(\bm{I}_n-\bm{V}_{\rm{s}}\bm{V}_{\rm{s}}^*)=\bm{O}_{m\times n}$.

Consequently, we conclude that conditions (i), (ii) and (iii) are equivalent.
\end{proof}

A dual complex matrix $\bm{A}$ is said to be of rank $r$ if and only if the rank of $\bm{A}_{\rm{s}}$ is $r$ and the dual rank-$r$ decomposition of $\bm{A}$ exists.
Thus, \Cref{equiva} confirms our previous conjecture that a rank-$k$ dual complex matrix $\bm{A}_k$ has the form $\bm{U}_k\bm{\Sigma}_k \bm{V}_k^*$, where $\bm{U}_k = \bm{U}_{\rm{s}}^{(k)}+\bm{U}_{\rm{i}}^{(k)}\epsilon\in\mathbb{DC}^{m\times k}$, $\bm{V}_k = \bm{V}_{\rm{s}}^{(k)}+\bm{V}_{\rm{i}}^{(k)}\epsilon\in\mathbb{DC}^{n\times k}$ have unitary columns, and $\bm{\Sigma}_k = \bm{\Sigma}_{\rm{s}}^{(k)}+\bm{\Sigma}_{\rm{i}}^{(k)}\epsilon\in\mathbb{DR}^{k\times k}$ is diagonal positive. 
Here, our goal is to find the optimal rank-$k$ approximation for the original dual matrix. 
One aspect is to introduce the total order of dual numbers.
Given two dual numbers $p = p_{\rm{s}} + p_{\rm{i}}\epsilon$, $q = q_{\rm{s}} + q_{\rm{i}}\epsilon$, where $p_{\rm{s}}$, $p_{\rm{i}}$, $q_{\rm{s}}$ and $q_{\rm{i}}$ are all real numbers, we say that $p<  q$, if either $p_{\rm{s}} < q_{\rm{s}}$, or $p_{\rm{s}} = q_{\rm{s}}$ and $p_{\rm{i}}<q_{\rm{i}}$ \cite{qi2022dual}.
Another aspect is to suggest a reasonable measure of the distance between two dual matrices, such as metric or quasi-metric.
The following optimization problem shows the rank-$k$ approximation of a dual complex matrix $\bm{A}=\bm{A}_{\rm{s}}+\bm{A}_{\rm{i}}\epsilon\in\mathbb{DC}^{m\times n}$ under a given metric or quasi-metric $d:\mathbb{DC}^{m\times n}\times\mathbb{DC}^{m\times n}\rightarrow \mathbb{D}$ ($k$ is less than or equal to the rank of $\bm{A}_{\rm{s}}$):
\begin{align}
\begin{array}{cl}
\min\limits_{\bm{U}_k,\bm{V}_k,\bm{\Sigma}_k} & d^2\big( \bm{A}\;,\;\bm{U}_k\bm{\Sigma}_k \bm{V}_k^*\big)\;,    \\
    s.t. & \bm{U}_k^*\bm{U}_k = \bm{I}_k\;,\;
    \bm{V}_k^*\bm{V}_k = \bm{I}_k\;,\\
    & (\bm{\Sigma}_k)_{ij}=0\;(\forall i\neq j)\;, \; (\bm{\Sigma}_k)_{ii}>0\;(\forall i=1,\cdots,k) \;,\\
    & \bm{U}_k\in\mathbb{DC}^{m\times k}\;,\;\bm{V}_k\in\mathbb{DC}^{n\times k}\;,\;\bm{\Sigma}_k\in\mathbb{DR}^{k\times k}\;.
    \end{array}\label{17}
\end{align}

Here, the definition of quasi-metric mapping to the dual number is inspired by Rainio's definition of quasi-metric \cite{rainio2021intrinsic} that maps to the real number.
\begin{definition}\label{def:quasi-metric}
    For the space of dual complex matrices $\mathbb{DC}^{m\times n}$, a quasi-metric is a function $d:\mathbb{DC}^{m\times n}\times \mathbb{DC}^{m\times n}\rightarrow \mathbb{D}$ that fulfills the following three conditions for all $\bm{A}$, $\bm{B}$, $\bm{C}\in\mathbb{DC}^{m\times n}$, based on the total order of dual numbers:
    \begin{enumerate}
        \item[${\rm{(i)}}$] Positivity: $d(\bm{A},\bm{B})\geq 0$, and $d(\bm{A},\bm{B})= 0$ if and only if $\bm{A}=\bm{B}$.
        \item[${\rm{(ii)}}$] Symmetry: $d(\bm{A},\bm{B})=d(\bm{B},\bm{A})$.
        \item[${\rm{(iii)}}$] Weaker triangle  inequality: $d(\bm{A},\bm{B})\leq c\big(d(\bm{A},\bm{C})+d(\bm{C},\bm{B})\big)$ with some constant dual number $c>1$ independent of the variables $\bm{A}, \bm{B}, \bm{C}$.
    \end{enumerate}
\end{definition}

\subsection{\texorpdfstring{Rank-$k$ Approximation Under the Metric Induced by Frobenius Norm}{Rank-k Approximation Under the Metric Induced by Frobenius Norm}} 
Qi et al. \cite{qi2022low} proposed an Eckart-Young-like theorem revealing that the sum of the first $k$ components of the SVD for a dual complex matrix is the closest to the original one compared to any other dual matrices with rank at most $k$. \Cref{de3.1} shows the corresponding norm. 
\begin{definition}[\cite{qi2022low}]\label{de3.1}
Let $\bm{A} = \bm{A}_{\rm{s}}+\bm{A}_{\rm{i}}\epsilon\in\mathbb{DC}^{m\times n}$. Define the Frobenius norm of $\bm{A}$ as
\begin{align}
    \Vert \bm{A}\Vert_{F} =\left\{\begin{array}{cc}
         \Vert \bm{A}_{\rm{s}}\Vert_F+\frac{\langle \bm{A}_{\rm{s}},\bm{A}_{\rm{i}}\rangle}{\Vert \bm{A}_{\rm{s}}\Vert_F}\epsilon\;, & if \; \bm{A}_{\rm{s}}\neq \bm{O}_{m\times n}\;,  \\
          \Vert \bm{A}_{\rm{i}}\Vert_F\epsilon\;, & otherwise\;,
    \end{array} \right.
\end{align}
where the inner product of matrices is over the real number field, that is, $\langle \bm{A}_{\rm{s}},\bm{A}_{\rm{i}}\rangle = Re\big(trace(\bm{A}_{\rm{s}}^*\bm{A}_{\rm{i}})\big)$.
\end{definition}

However, we have found that under the induced metric, the infinitesimal part of the optimal approximate dual matrix $\bm{A}_k$ is arbitrary, as long as the standard part of $\bm{A}_k$ is the best approximation of the standard part of the original dual matrix $\bm{A}$.
\Cref{the_low_1} verifies this.
\begin{proposition}\label{the_low_1}
Let $\bm{A}=\bm{A}_{\rm{s}}+\bm{A}_{\rm{i}}\epsilon\in\mathbb{DC}^{m\times n}$. Then under the metric induced by Frobenius norm of dual complex matrices, the optimization problem \cref{17} always achieves a minimum value for any $\bm{U}_{\rm{i}}^{(k)},\bm{V}_{\rm{i}}^{(k)},\bm{\Sigma}_{\rm{i}}^{(k)}$, as long as $\bm{U}_{\rm{s}}^{(k)}\bm{\Sigma}_{\rm{s}}^{(k)}\big(\bm{V}_{\rm{s}}^{(k)}\big)^*$ is the best rank-$k$ approximation of $\bm{A}_{\rm{s}}$.
\end{proposition}

\begin{proof}
Let $\bm{E}_{\rm{s}}^{(k)}:=\bm{A}_{\rm{s}}-\bm{U}_{\rm{s}}^{(k)}\bm{\Sigma}_{\rm{s}}^{(k)}\big(\bm{V}_{\rm{s}}^{(k)}\big)^*$ and $\bm{E}_{\rm{i}}^{(k)}:= \bm{A}_{\rm{i}}-\bm{U}_{\rm{s}}^{(k)}\bm{\Sigma}_{\rm{s}}^{(k)}\big(\bm{V}_{\rm{i}}^{(k)}\big)^*-\bm{U}_{\rm{s}}^{(k)}\bm{\Sigma}_{\rm{i}}^{(k)} \big(\bm{V}_{\rm{s}}^{(k)}\big)^*-\bm{U}_{\rm{i}}^{(k)}\bm{\Sigma}_{\rm{s}}^{(k)}\big(\bm{V}_{\rm{s}}^{(k)}\big)^*$. Consider the metric induced by Frobenius norm
\begin{align}
d_F\big(\bm{A},\bm{U}_k\bm{\Sigma}_k \bm{V}_k^*\big) = \Vert \bm{A}-\bm{U}_k\bm{\Sigma}_k \bm{V}_k^*\Vert_{F}\;,\nonumber
\end{align}
then the objective function \cref{17} becomes
\begin{equation}
d_F^2\big( \bm{A},\bm{U}_k\bm{\Sigma}_k \bm{V}_k^*\big)
= \big\Vert \bm{E}_{\rm{s}}^{(k)}+\bm{E}_{\rm{i}}^{(k)}\epsilon \big\Vert^2_{F}
= \big\Vert \bm{E}_{\rm{s}}^{(k)}\big\Vert_{F}^2+2\big\langle \bm{E}_{\rm{s}}^{(k)},\bm{E}_{\rm{i}}^{(k)}\big\rangle\epsilon \label{18}\;.
\end{equation}

According to the total order ``$<$'' of dual numbers, we first minimize the standard part $\big\Vert\bm{E}_{\rm{s}}^{(k)}\big\Vert_{F}^2$. When $\bm{U}_{\rm{s}}^{(k)}\bm{\Sigma}_{\rm{s}}^{(k)}\big(\bm{V}_{\rm{s}}^{(k)}\big)^*$ is the best rank-$k$ approximation of complex matrix $\bm{A}_{\rm{s}}$, $\big\Vert\bm{E}_{\rm{s}}^{(k)}\big\Vert_{F}^2$ achieves the minimum value, at this time, we also have
\begin{align}
    \big(\bm{U}_{\rm{s}}^{(k)}\big)^*\bm{E}_{\rm{s}}^{(k)} =\bm{O}_{r\times n}\;,\;
    \bm{E}_{\rm{s}}^{(k)}\bm{V}_{\rm{s}}^{(k)}=\bm{O}_{m\times r}\;\nonumber.
\end{align}

Hence, the expression \cref{18} can be organized as
\begin{align*}
    &d_F^2\big( \bm{A},\bm{U}_k\bm{\Sigma}_k \bm{V}_k^*\big)\nonumber\\
    =& \big\Vert \bm{E}_{\rm{s}}^{(k)}\big\Vert_{F}^2+2\big\langle \bm{E}_{\rm{s}}^{(k)},\bm{A}_{\rm{i}}\big\rangle\epsilon-2\big\langle \big(\bm{U}_{\rm{s}}^{(k)}\big)^*\bm{E}_{\rm{s}}^{(k)},\bm{\Sigma}_{\rm{s}}^{(k)}\big(\bm{V}_{\rm{i}}^{(k)}\big)^*\big\rangle\epsilon\\
    &\quad\quad\quad\quad-2\big\langle \bm{E}_{\rm{s}}^{(k)}\bm{V}_{\rm{s}}^{(k)},\bm{U}_{\rm{s}}^{(k)}\bm{\Sigma}_{\rm{i}}^{(k)} \big\rangle\epsilon-2\big\langle \bm{E}_{\rm{s}}^{(k)}\bm{V}_{\rm{s}}^{(k)},\bm{U}_{\rm{i}}^{(k)}\bm{\Sigma}_{\rm{s}}^{(k)}\big\rangle\epsilon\nonumber\\
    =& \big\Vert \bm{E}_{\rm{s}}^{(k)}\big\Vert_F^2+2\big\langle \bm{E}_{\rm{s}}^{(k)}, \bm{A}_{\rm{i}}\big\rangle\epsilon\;\nonumber,
\end{align*}
which implies that $d_F^2( \bm{A},\bm{U}_k\bm{\Sigma}_k \bm{V}_k^*)$ is a fixed value for any $\bm{U}_{\rm{i}}^{(k)},\bm{V}_{\rm{i}}^{(k)},\bm{\Sigma}_{\rm{i}}^{(k)}$, as long as $\bm{U}_{\rm{s}}^{(k)}\bm{\Sigma}_{\rm{s}}^{(k)}\big(\bm{V}_{\rm{s}}^{(k)}\big)^*$ is the best rank-$k$ approximation of the complex matrix $\bm{A}_{\rm{s}}$. This completes the proof.
\end{proof}

\subsection{\texorpdfstring{Rank-$k$ Approximation Under the Quasi-Metric $d_*$}{Rank-k Approximation Under the Quasi-Metric d*}}
In order to interpret the role of the infinitesimal part in low-rank approximation, we define a quasi-metric to measure the distance between two dual matrices. \begin{definition}\label{def3.2}
Let $\bm{A}=\bm{A}_{\rm{s}}+\bm{A}_{\rm{i}}\epsilon,\bm{B}=\bm{B}_{\rm{s}}+\bm{B}_{\rm{i}}\epsilon\in\mathbb{DC}^{m\times n}$. Define a function $d_*:\mathbb{DC}^{m\times n}\times\mathbb{DC}^{m\times n}\rightarrow\mathbb{D}$ satisfying
\begin{align}
d_*( \bm{A},\bm{B}) =\left\{\begin{array}{cc}
\Vert \bm{A}_{\rm{s}}-\bm{B}_{\rm{s}}\Vert_F+\frac{\Vert \bm{A}_{\rm{i}}-\bm{B}_{\rm{i}}\Vert_F^2}{ 2\Vert \bm{A}_{\rm{s}}-\bm{B}_{\rm{s}}\Vert_F}\epsilon\;, & if \; \bm{A}_{\rm{s}}\neq\bm{B}_{\rm{s}}\;,  \\
\Vert \bm{A}_{\rm{i}}-\bm{B}_{\rm{i}}\Vert_F\epsilon\;, & otherwise\;.
\end{array} \right.\label{quasi-metric}
\end{align}
\end{definition}

Actually, the above function $d_*$ is a quasi-metric, and the first two conditions in \Cref{def:quasi-metric} apparently hold. 
In addition, take the constant dual number $c=c_{\rm{s}}+c_{\rm{i}}\epsilon$ satisfying $c_{\rm{s}}>1$, then the third condition in \Cref{def:quasi-metric} is always true. 
Since the standard part of $d_*$ fulfills that $\Vert \bm{A}_{\rm{s}}-\bm{B}_{\rm{s}}\Vert_F\leq\Vert \bm{A}_{\rm{s}}-\bm{C}_{\rm{s}}\Vert_F+\Vert \bm{C}_{\rm{s}}-\bm{B}_{\rm{s}}\Vert_F$ for any dual complex matrices $\bm{A}, \bm{B}, \bm{C}$, we obtain $\Vert \bm{A}_{\rm{s}}-\bm{B}_{\rm{s}}\Vert_F<c_{\rm{s}}\big(\Vert \bm{A}_{\rm{s}}-\bm{C}_{\rm{s}}\Vert_F+\Vert \bm{C}_{\rm{s}}-\bm{B}_{\rm{s}}\Vert_F\big)$, which implies $d_*(\bm{A},\bm{B})< c\big(d_*(\bm{A},\bm{C})+d_*(\bm{C},\bm{B})\big)$ based on the total order of dual numbers.

Incidentally, $d_*$ is not a metric because it does not satisfy the triangle inequality. For example, let $\bm{A} = \left(\begin{smallmatrix}
        1 & 0 \\
            0 & 1 \\
    \end{smallmatrix}\right) + \left(\begin{smallmatrix}
        1 & 0 \\
            0 & 1 \\
\end{smallmatrix}\right)\epsilon$ , $\bm{B} = \left(\begin{smallmatrix}
        0 & 0 \\
            0 & 0 \\
    \end{smallmatrix}\right) + \left(\begin{smallmatrix}
        -1 & 0 \\
            0 & -1 \\  \end{smallmatrix}\right)\epsilon$ and $\bm{C} = \left(\begin{smallmatrix}
        0 & 0 \\
            0 & 0 \\
    \end{smallmatrix}\right) + \left(\begin{smallmatrix}
        0 & 0 \\
            0 & 0 \\  \end{smallmatrix}\right)\epsilon$. Then
            \begin{align}
                d_*(\bm{A},\bm{B}) &= \Vert \left(\begin{smallmatrix}
        1 & 0 \\
            0 & 1 \\
\end{smallmatrix}\right)\Vert_F+\frac{\Vert \left(\begin{smallmatrix}
        2 & 0 \\
            0 & 2 \\
\end{smallmatrix}\right)\Vert_F^2}{2\Vert \left(\begin{smallmatrix}
        1 & 0 \\
            0 & 1 \\
\end{smallmatrix}\right)\Vert_F}\epsilon = \sqrt{2}+2\sqrt{2}\epsilon\;,\nonumber\\
d_*(\bm{A},\bm{C}) &= \Vert \left(\begin{smallmatrix}
        1 & 0 \\
            0 & 1 \\
\end{smallmatrix}\right)\Vert_F+\frac{\Vert \left(\begin{smallmatrix}
        1 & 0 \\
            0 & 1 \\
\end{smallmatrix}\right)\Vert_F^2}{2\Vert \left(\begin{smallmatrix}
        1 & 0 \\
            0 & 1 \\
\end{smallmatrix}\right)\Vert_F}\epsilon = \sqrt{2}+\frac{\sqrt{2}}{2}\epsilon\;,\nonumber\\
d_*(\bm{C},\bm{B}) &= \Vert \left(\begin{smallmatrix}
        1 & 0 \\
            0 & 1 \\
\end{smallmatrix}\right)\Vert_F\epsilon = \sqrt{2}\epsilon\;.\nonumber
            \end{align}
However, $d_*(\bm{A},\bm{B})>d_*(\bm{A},\bm{C})+d_*(\bm{C},\bm{B})$. 

Moreover, it is reasonable to use the quasi-metric $d_*$ in the optimization problem \cref{17}. 
Since according to the total order of dual numbers, minimizing the objective function is equivalent to first minimizing its standard part and then its infinitesimal part.
This process is similar to multi-objective optimization.

The following theorem states the optimal rank-$k$ approximation of a given dual complex matrix under the quasi-metric $d_*$, which reflects the impact of the infinitesimal part. 

\begin{theorem}
Let $\bm{A}=\bm{A}_{\rm{s}}+\bm{A}_{\rm{i}}\epsilon\in\mathbb{DC}^{m\times n} (m\geq n)$. Then under the quasi-metric $d_*$ \cref{quasi-metric}, the optimization problem \cref{17} has the optimal solution $\big(\widehat{\bm{U}}_k,\widehat{\bm{V}}_k,\widehat{\bm{\Sigma}}_k\big)$, where $\widehat{\bm{U}}_k = \widehat{\bm{U}}_{\rm{s}}^{(k)}+\widehat{\bm{U}}_{\rm{i}}^{(k)}\epsilon\in\mathbb{DC}^{m\times k}$, $\widehat{\bm{V}}_k = \widehat{\bm{V}}_{\rm{s}}^{(k)}+\widehat{\bm{V}}_{\rm{i}}^{(k)}\epsilon\in\mathbb{DC}^{n\times k}$ have unitary columns and $\widehat{\bm{\Sigma}}_k = \widehat{\bm{\Sigma}}_{\rm{s}}^{(k)}+\widehat{\bm{\Sigma}}_{\rm{i}}^{(k)}\epsilon\in\mathbb{DR}^{k\times k}$ is diagonal positive. Among them, the standard parts satisfy that $\widehat{\bm{U}}_{\rm{s}}^{(k)}\widehat{\bm{\Sigma}}_{\rm{s}}^{(k)}\big(\widehat{\bm{V}}_{\rm{s}}^{(k)}\big)^*$ is the best rank-$k$ approximation of $\bm{A}_{\rm{s}}$, where $\widehat{\bm{\Sigma}}_{\rm{s}}^{(k)}={\rm{diag}}(\hat{\sigma}_1\bm{I}_{\hat{r}_1},\hat{\sigma}_2\bm{I}_{\hat{r}_2},\cdots,\hat{\sigma}_\iota\bm{I}_{\hat{r}_\iota})$ is diagonal positive with $\hat{\sigma}_1>\hat{\sigma}_2>\cdots>\hat{\sigma}_\iota>0$. Moreover, the infinitesimal parts satisfy
\begin{subequations}
\begin{align}
\widehat{\bm{U}}_{\rm{i}}^{(k)}&=\widehat{\bm{U}}_{\rm{s}}^{(k)}\left[{\rm{sym}}\big(\bm{\widehat{R}}\widehat{\bm{\Sigma}}_{\rm{s}}^{(k)}\big)\odot\bm{\widehat{\Delta}}+\bm{\widehat{\Omega}}+\bm{\widehat{\Psi}}\right]+\mathcal{P}_{\big(\widehat{\bm{U}}_{\rm{s}}^{(k)}\big)^\perp} \bm{A}_{\rm{i}}\widehat{\bm{V}}_{\rm{s}}^{(k)}\big(\widehat{\bm{\Sigma}}_{\rm{s}}^{(k)}\big)^{-1}\;,\label{20a}\\
\widehat{\bm{V}}_{\rm{i}}^{(k)}&=\widehat{\bm{V}}_{\rm{s}}^{(k)}\left[{\rm{sym}}\big(\widehat{\bm{\Sigma}}_{\rm{s}}^{(k)}\bm{\widehat{R}}\big)\odot\bm{\widehat{\Delta}}+\bm{\widehat{\Omega}}\right]+\mathcal{P}_{\big(\widehat{\bm{V}}_{\rm{s}}^{(k)}\big)^\perp}  \bm{A}_{\rm{i}}^*\widehat{\bm{U}}_{\rm{s}}^{(k)}\big(\widehat{\bm{\Sigma}}_{\rm{s}}^{(k)}\big)^{-1}\;,\label{20b}\\
\widehat{\bm{\Sigma}}_{\rm{i}}^{(k)}&=\frac{1}{2}{\rm{Diag}}\Big[{\rm{sym}}\big(\bm{\widehat{R}}\big)\Big]\;,\label{20c}
\end{align}
\end{subequations}
where $\bm{\widehat{R}}:=\big(\widehat{\bm{U}}_{\rm{s}}^{(k)}\big)^*\bm{A}_{\rm{i}}\widehat{\bm{V}}_{\rm{s}}^{(k)}$ and $\bm{\widehat{R}}_{ii}$ represents the $(i,i)$-block matrix of $\bm{\widehat{R}}$ partitioned as in $\widehat{\bm{\Sigma}}_{\rm{s}}^{(k)}$, $\bm{\widehat{\Omega}}\in\mathcal{B}_{k\iota}^{\mathcal{S}}$ and 
\begin{align}
\bm{\widehat{\Delta}} = \begin{bmatrix}
    \bm{O}_{\hat{r}_1}  & \frac{\bm{1}_{\hat{r}_1\times \hat{r}_2}}{\hat{\sigma}_2^2-\hat{\sigma}_1^2} & \cdots & \frac{\bm{1}_{\hat{r}_1\times \hat{r}_\iota}}{\hat{\sigma}_\iota^2-\hat{\sigma}_1^2} \\
       \frac{\bm{1}_{\hat{r}_2\times \hat{r}_1}}{\hat{\sigma}_1^2-\hat{\sigma}_2^2}  & \bm{O}_{\hat{r}_2} & \cdots & \frac{\bm{1}_{\hat{r}_2\times \hat{r}_\iota}}{\hat{\sigma}_\iota^2-\hat{\sigma}_2^2} \\
       \vdots & \vdots & \ddots & \vdots \\
       \frac{\bm{1}_{\hat{r}_\iota\times \hat{r}_1}}{\hat{\sigma}_1^2-\hat{\sigma}_\iota^2} & \frac{\bm{1}_{\hat{r}_\iota\times \hat{r}_2}}{\hat{\sigma}_2^2-\hat{\sigma}_\iota^2} & \cdots & \bm{O}_{\hat{r}_\iota}
\end{bmatrix},
    \widehat{\bm{\Psi}} = \begin{bmatrix}
        \frac{\bm{\widehat{R}}_{11}-\bm{\widehat{R}}_{11}^*}{2\hat{\sigma}_1}    & & \\
          & \ddots & \\
          & & \frac{\bm{\widehat{R}}_{\iota\iota}-\bm{\widehat{R}}_{\iota\iota}^*}{2\hat{\sigma}_\iota}
    \end{bmatrix}.\label{29}
\end{align}
\end{theorem}
\begin{proof}
Let $\bm{E}_{\rm{s}}^{(k)}:=\bm{A}_{\rm{s}}-\bm{U}_{\rm{s}}^{(k)}\bm{\Sigma}_{\rm{s}}^{(k)}\big(\bm{V}_{\rm{s}}^{(k)}\big)^*$ and $\bm{E}_{\rm{i}}^{(k)}:=\bm{A}_{\rm{i}}-\bm{U}_{\rm{s}}^{(k)}\bm{\Sigma}_{\rm{s}}^{(k)}\big(\bm{V}_{\rm{i}}^{(k)}\big)^*-\bm{U}_{\rm{s}}^{(k)}\bm{\Sigma}_{\rm{i}}^{(k)} \big(\bm{V}_{\rm{s}}^{(k)}\big)^*-\bm{U}_{\rm{i}}^{(k)}\bm{\Sigma}_{\rm{s}}^{(k)}\big(\bm{V}_{\rm{s}}^{(k)}\big)^*$. Then under the quasi-metric $d_*$ \cref{quasi-metric}, the objective function \cref{17} becomes
\begin{align}
d_*^2\big( \bm{A},\bm{U}_k\bm{\Sigma}_k \bm{V}_k^*\big)=   \big\Vert \bm{E}_{\rm{s}}^{(k)}\big\Vert_{F}^2+\big\Vert \bm{E}_{\rm{i}}^{(k)}\big\Vert_F^2\epsilon \label{30}\;.
\end{align}

In the light of the total order $``<"$ of dual numbers, we start with minimizing the standard part $\big\Vert \bm{E}_{\rm{s}}^{(k)}\big\Vert_{F}^2$. 
Assume that $\bm{A}_{\rm{s}} = \bm{M}_{\rm{s}}\bm{\Sigma}_{\rm{s}}\bm{N}_{\rm{s}}^*$ is a compact SVD of $\bm{A}_{\rm{s}}$. Let $\bm{\widehat{R}} = \bm{M}_{\rm{s}}^*\bm{A}_{\rm{i}}\bm{N}_{\rm{s}}$ and $\bm{\widehat{R}}_{ii}$ represents the $(i,i)$-block matrix of $\bm{\widehat{R}}$ partioned as in $\bm{\Sigma}_{\rm{s}}$. Compute the eigenvalue decomposition of $\frac{\bm{\widehat{R}}_{ii}+\bm{\widehat{R}}_{ii}^*}{2}$ such that $\frac{\bm{\widehat{R}}_{ii}+\bm{\widehat{R}}_{ii}^*}{2} = \bm{X}_{ii}\bm{\Lambda}_{ii}\bm{X}_{ii}^*$, where $\bm{X}_{ii}$ is unitary and $\bm{\Lambda}_{ii}$ is diagonal real. $\bm{X}$ is obtained by diagonally splicing all the $\bm{X}_{ii}$. Set $\bm{U}_{\rm{s}}=\bm{M}_{\rm{s}}\bm{X}$ and $\bm{V}_{\rm{s}}=\bm{N}_{\rm{s}}\bm{X}$, then $\bm{U}_{\rm{s}}\bm{\Sigma}_{\rm{s}}\bm{V}_{\rm{s}}^*$ is also a compact SVD of $\bm{A}_{\rm{s}}$. When $\widehat{\bm{U}}_{\rm{s}}^{(k)} = \bm{U}_{\rm{s}}(:,1:k)$, $\bm{\widehat{\Sigma}}_{\rm{s}}^{(k)} = \bm{\Sigma}_{\rm{s}}(1:k,1:k)$ and $\bm{\widehat{V}}_{\rm{s}}^{(k)} = \bm{V}_{\rm{s}}(:,1:k)$ are substituted for $\bm{U}_{\rm{s}}^{(k)}$, $\bm{\Sigma}_{\rm{s}}^{(k)}$ and $\bm{V}_{\rm{s}}^{(k)}$, $\big\Vert \bm{E}_{\rm{s}}^{(k)}\big\Vert_{F}^2$ achieves the minimum value.

We then minimize the infinitesimal part of \cref{30} by using the above optimal pair $(\widehat{\bm{U}}_{\rm{s}}^{(k)},\widehat{\bm{\Sigma}}_{\rm{s}}^{(k)},\widehat{\bm{V}}_{\rm{s}}^{(k)})$. On account of original constraint conditions, the initial optimization problem \cref{17} is equivalent to
\begin{align}
\begin{array}{cl}
\min\limits_{\bm{U}_{\rm{i}}^{(k)},\bm{V}_{\rm{i}}^{(k)},\bm{\Sigma}_{\rm{i}}^{(k)}} & \big\Vert \bm{E}_{\rm{i}}^{(k)}\big\Vert_F^2\;,    \\
s.t. & \big(\widehat{\bm{U}}_{\rm{s}}^{(k)}\big)^*\bm{U}_{\rm{i}}^{(k)}+\big(\bm{U}_{\rm{i}}^{(k)}\big)^*\widehat{\bm{U}}_{\rm{s}}^{(k)} = \bm{O}_k\;,\\
&\big(\widehat{\bm{V}}_{\rm{s}}^{(k)}\big)^*\bm{V}_{\rm{i}}^{(k)}+\big(\bm{V}_{\rm{i}}^{(k)}\big)^*\widehat{\bm{V}}_{\rm{s}}^{(k)} = \bm{O}_k\;,\\
& \big(\bm{\Sigma}_{\rm{i}}^{(k)}\big)_{ij}=0\; (\forall i\neq j)\;,\\
& \bm{U}_{\rm{i}}^{(k)}\in\mathbb{C}^{m\times k}\;,\;\bm{V}_{\rm{i}}^{(k)}\in\mathbb{C}^{n\times k}\;,\;\bm{\Sigma}_{\rm{i}}^{(k)}\in\mathbb{R}^{k\times k}\;.
\end{array}\label{30000}
\end{align}

Based on \Cref{pro1.2}, suppose that $\bigl[\widehat{\bm{U}}_{\rm{s}}^{(k)}\;\;\Breve{\bm{U}}_{\rm{s}}\bigr]\in\mathbb{C}^{m\times m}$ and $\bigl[\widehat{\bm{V}}_{\rm{s}}^{(k)}\;\;\Breve{\bm{V}}_{\rm{s}}\bigr]\in\mathbb{C}^{n\times n}$ are both unitary. 
Applying the unitary invariant property of the Frobenius norm, we obtain
\begin{align}
&\big\Vert \bm{E}_{\rm{i}}^{(k)}\big\Vert_F^2\nonumber\\
=&\left\Vert
    \bigg[\begin{array}{c}
       \big(\widehat{\bm{U}}_{\rm{s}}^{(k)}\big)^*    \\
       \Breve{\bm{U}}_{\rm{s}}^* 
    \end{array}\bigg]\bm{E}_{\rm{i}}^{(k)}\Big[\begin{array}{cc}
        \widehat{\bm{V}}_{\rm{s}}^{(k)} & \Breve{\bm{V}}_{\rm{s}} \\
    \end{array}\Big]\right\Vert_F^2\nonumber\\
    =&\left\Vert
    \bigg[\begin{array}{cc}
       \big(\widehat{\bm{U}}_{\rm{s}}^{(k)}\big)^*\bm{A}_{\rm{i}}\widehat{\bm{V}}_{\rm{s}}^{(k)}   & \big(\widehat{\bm{U}}_{\rm{s}}^{(k)}\big)^*\bm{A}_{\rm{i}}\Breve{\bm{V}}_{\rm{s}}\\
       \Breve{\bm{U}}_{\rm{s}}^*\bm{A}_{\rm{i}}\widehat{\bm{V}}_{\rm{s}}^{(k)} & \Breve{\bm{U}}_{\rm{s}}^*\bm{A}_{\rm{i}}\Breve{\bm{V}}_{\rm{s}} \\
    \end{array}\bigg]-\bigg[\begin{array}{cc}
       \bm{M}^{(k)}   &  \widehat{\bm{\Sigma}}_{\rm{s}}^{(k)}\big(\bm{V}_{\rm{i}}^{(k)}\big)^*\Breve{\bm{V}}_{\rm{s}}\\
       \Breve{\bm{U}}_{\rm{s}}^*\bm{U}_{\rm{i}}^{(k)}\widehat{\bm{\Sigma}}_{\rm{s}}^{(k)} & \bm{O}_{(m-k)\times (n-k)} \\
    \end{array}\bigg]\right\Vert_F^2\nonumber\\
    =& \left\Vert\Breve{\bm{U}}_{\rm{s}}^*\bm{A}_{\rm{i}}\Breve{\bm{V}}_{\rm{s}}\right\Vert_F^2 + \left\Vert\big(\widehat{\bm{U}}_{\rm{s}}^{(k)}\big)^*\bm{A}_{\rm{i}}\widehat{\bm{V}}_{\rm{s}}^{(k)}-\bm{M}^{(k)}\right\Vert_F^2+\left\Vert\big(\widehat{\bm{U}}_{\rm{s}}^{(k)}\big)^*\bm{A}_{\rm{i}}\Breve{\bm{V}}_{\rm{s}}-\widehat{\bm{\Sigma}}_{\rm{s}}^{(k)}\big(\bm{V}_{\rm{i}}^{(k)}\big)^*\Breve{\bm{V}}_{\rm{s}}\right\Vert_F^2\nonumber\\
    &\quad\quad\quad\quad\quad\;\;+\left\Vert\Breve{\bm{U}}_{\rm{s}}^*\bm{A}_{\rm{i}}\widehat{\bm{V}}_{\rm{s}}^{(k)}-\Breve{\bm{U}}_{\rm{s}}^*\bm{U}_{\rm{i}}^{(k)}\widehat{\bm{\Sigma}}_{\rm{s}}^{(k)}\right\Vert_F^2,\nonumber
\end{align}
where $\bm{M}^{(k)}:=\widehat{\bm{\Sigma}}_{\rm{s}}^{(k)}\big(\bm{V}_{\rm{i}}^{(k)}\big)^*\widehat{\bm{V}}_{\rm{s}}^{(k)}+\bm{\Sigma}_{\rm{i}}^{(k)}+\big(\widehat{\bm{U}}_{\rm{s}}^{(k)}\big)^*\bm{U}_{\rm{i}}^{(k)}\widehat{\bm{\Sigma}}_{\rm{s}}^{(k)}$.

When $\big(\widehat{\bm{U}}_{\rm{i}}^{(k)},\widehat{\bm{V}}_{\rm{i}}^{(k)},\widehat{\bm{\Sigma}}_{\rm{i}}^{(k)}\big)$ in \cref{20a,20b,20c} are substituted for $\big(\bm{U}_{\rm{i}}^{(k)},\bm{V}_{\rm{i}}^{(k)},\bm{\Sigma}_{\rm{i}}^{(k)}\big)$,
\Cref{lemma3.1} demonstrates that the second term of the above expression reaches the minimum value of zero since $\frac{\bm{R}_{ii}+\bm{R}_{ii}^*}{2}$ is diagonal real, where $\bm{R}=\bm{U}_{\rm{s}}\bm{\Sigma}_{\rm{s}}\bm{V}_{\rm{s}}^*$. Besides, the last two terms also attain their minimum values of zero from \Cref{le2.2}. Then the new optimization problem \cref{30000} achieves its minimum value $\big\Vert\Breve{\bm{U}}_{\rm{s}}^*\bm{A}_{\rm{i}}\Breve{\bm{V}}_{\rm{s}}\big\Vert_F^2$. Thus, we complete the proof.
\end{proof}

For a dual complex matrix failing the condition \cref{333333}, although it has neither CDSVD nor a definition of rank, its optimal rank-$k$ approximate dual complex matrix exists under our newly defined quasi-metric. 
On the other hand, a dual complex matrix $\bm{A}$ satisfying the condition \cref{333333} has both the CDSVD, $\bm{U\Sigma V}^*$,  and low-rank approximation under the quasi-metric $d_*$. 
However, unlike the classical SVD, the sum of the first $k$ terms of the CDSVD, $\bm{U}(:,1:k)\bm{\Sigma}(1:k,1:k)\bm{V}(:,1:k)^*$, is not the optimal rank-$k$ approximate dual complex matrix, $\widehat{\bm{U}}_k\widehat{\bm{\Sigma}}_k\widehat{\bm{V}}_k^*$. The latter uses only the first $k$ terms of the SVD of $\bm{A}_{\rm{s}}$ to calculate infinitesimal parts but the former uses all. 
Although only the particular pair $\big(\widehat{\bm{U}}^{(k)}_{\rm{s}},\widehat{\bm{V}}^{(k)}_{\rm{s}}\big)$ can ensure $\widehat{\bm{\Sigma}}^{(k)}_{\rm{i}}$ to be diagonal real and there are some degrees of freedom in $\widehat{\bm{U}}_{\rm{i}}^{(k)}$ and $\widehat{\bm{V}}_{\rm{i}}^{(k)}$, the optimal rank-$k$ approximate dual matrix is fixed for any pair, and
\begin{align}
\widehat{\bm{A}}_k=\widehat{\bm{U}}_k\widehat{\bm{\Sigma}}_k\widehat{\bm{V}}_k^*   =\widehat{\bm{U}}_{\rm{s}}^{(k)}\widehat{\bm{\Sigma}}_{\rm{s}}^{(k)}\big(\widehat{\bm{V}}_{\rm{s}}^{(k)}\big)^*+\Big(    \bm{A}_{\rm{i}} - \mathcal{P}_{\big(\widehat{\bm{U}}_{\rm{s}}^{(k)}\big)^{\perp}} \bm{A}_{\rm{i}}\mathcal{P}_{\big(\widehat{\bm{V}}_{\rm{s}}^{(k)}\big)^{\perp}}\Big)\epsilon\nonumber\;.
\end{align}


\Cref{alg:lowrank} shows how to find the optimal rank-$k$ approximation for a given dual complex matrix under the quasi-metric $d_*$ \cref{quasi-metric}.

\begin{algorithm}[ht]
\caption{The Optimal Rank-$k$ Approximation of $\bm{A}$ under the Quasi-Metric.}
\label{alg:lowrank}
\begin{algorithmic}
\REQUIRE{$\bm{A}=\bm{A}_{\rm{s}}+\bm{A}_{\rm{i}}\epsilon\in\mathbb{DC}^{m\times n}$ and the rank of approximation $k$\;$\big(k\leq {\rm{Rank}}(\bm{A}_{\rm{s}})\big)$.}
\STATE{\textbf{Step 1}. 
Decompose $\bm{A}_{\rm{s}}$ by thin SVD $\bm{A}_{\rm{s}} = \bm{U}_{\rm{s}}\bm{\Sigma}_{\rm{s}}\bm{V}_{\rm{s}}^*$, where $\bm{U}_{\rm{s}}\in\mathbb{C}^{m\times n}$ has unitary columns, $\bm{\Sigma}_{\rm{s}}\in\mathbb{R}^{m\times n}$ is rectangular diagonal and $\bm{V}_{\rm{s}}\in\mathbb{C}^{n\times n}$ is unitary.}
\STATE{\textbf{Step 2}.
    Set $\widehat{\bm{U}}_{\rm{s}}^{(k)} = \bm{U}_{\rm{s}}(:,1:k)$, $\widehat{\bm{\Sigma}}_{\rm{s}}^{(k)} = \bm{\Sigma}_{\rm{s}}(1:k,1:k)$ and $\widehat{\bm{V}}_{\rm{s}}^{(k)} = \bm{V}_{\rm{s}}(:,1:k)$. }
\STATE{\textbf{Step 3}.
   Compute $\bm{A}_{\rm{s}}^{(k)} = \widehat{\bm{U}}_{\rm{s}}^{(k)}\widehat{\bm{\Sigma}}_{\rm{s}}^{(k)}\big(\widehat{\bm{V}}_{\rm{s}}^{(k)}\big)^*$.}
\STATE{\textbf{Step 4}.
    Compute $\bm{A}_{\rm{i}}^{(k)}=\bm{A}_{\rm{i}} - \Big(\bm{I}_m- \widehat{\bm{U}}_{\rm{s}}^{(k)}\big(\widehat{\bm{U}}_{\rm{s}}^{(k)}\big)^*\Big) \bm{A}_{\rm{i}}\Big(\bm{I}_n- \widehat{\bm{V}}_{\rm{s}}^{(k)}\big(\widehat{\bm{V}}_{\rm{s}}^{(k)}\big)^*\Big)$.}
\ENSURE{The optimal rank-$k$ approximate dual matrix of $\bm{A}$, $\bm{A}_k = \bm{A}_{\rm{s}}^{(k)} +\bm{A}_{\rm{i}}^{(k)}\epsilon$.}
\end{algorithmic}
\end{algorithm}

\section{Dual Moore-Penrose Generalized Inverse}\label{sec:DMPGI}
In addition to the low rank approximation, the CDSVD has many other applications similar to the SVD of matrices. Here, we take the dual Moore-Penrose generalized inverse (DMPGI) for example.

The existence of various types of dual generalized inverses of dual matrices has been discussed for a long time, especially for the DMPGI \cite{angeles2012dual,qi2023moore,udwadia2021does,udwadia2020all}. Here, we use a more intuitive and convenient method -- starting from the CDSVD -- to obtain the necessary and sufficient conditions for the existence of the DMPGI and its expression.

\begin{definition}[\cite{pennestri2009linear}]\label{defi:DMPGI}
Let $\bm{A}\in\mathbb{DC}^{m\times n}$, if there exists an $n$-by-$m$ dual complex matrix $\bm{X}$ satisfying ${\rm{(i)}} \bm{XAX}=\bm{X}$; ${\rm{(ii)}} \bm{AXA}=\bm{A}$; ${\rm{(iii)}} (\bm{XA})^*=\bm{XA}$; ${\rm{(iv)}} (\bm{AX})^*=\bm{AX}$, then we call $\bm{X}$ the dual Moore-Penrose generalized inverse (DMPGI) of $\bm{A}$, and denote it by $\bm{A}^\dag$.
\end{definition}

In the following theorem, we aim to provide some equivalent conditions for the existence of the DMPGI.

\begin{theorem}
Let $\bm{A}=\bm{A}_{\rm{s}}+\bm{A}_{\rm{i}}\epsilon\in\mathbb{DC}^{m\times n} (m\geq n)$. Assume that $\bm{A}_{\rm{s}} = \bm{U}_{\rm{s}}\bm{\Sigma}_{\rm{s}} \bm{V}_{\rm{s}}^*$ is a compact SVD of $\bm{A}_{\rm{s}}$, where $\bm{U}_{\rm{s}}\in\mathbb{C}^{m\times r}$, $\bm{V}_{\rm{s}}\in\mathbb{C}^{n\times r}$ both have unitary columns, and $\bm{\Sigma}_{\rm{s}}\in\mathbb{R}^{r\times r}$ is diagonal positive. Then the following conditions are equivalent:
\begin{enumerate}
    \item[${\rm{(i)}}$] $(\bm{I}_m-\bm{U}_{\rm{s}}\bm{U}_{\rm{s}}^*)\bm{A}_{\rm{i}}(\bm{I}_n-\bm{V}_{\rm{s}}\bm{V}_{\rm{s}}^*)=\bm{O}_{m\times n}$;
    \item[${\rm{(ii)}}$] The CDSVD of $\bm{A}$ exists;
    \item[${\rm{(iii)}}$] The DMPGI $\bm{A}^{\dag}$ of $\bm{A}$ exists, and 
    \begin{align}
        \bm{A}^{\dag} =  \bm{V}_{\rm{s}}\bm{\Sigma}_{\rm{s}}^{-1}\bm{U}_{\rm{s}}^*&+\Big[(\bm{I}_n-\bm{V}_{\rm{s}}\bm{V}_{\rm{s}}^*)\bm{A}_{\rm{i}}^*\bm{U}_{\rm{s}}\bm{\Sigma}_{\rm{s}}^{-2}\bm{U}_{\rm{s}}^*+\bm{V}_{\rm{s}}\bm{\Sigma}_{\rm{s}}^{-2}\bm{V}_{\rm{s}}^*\bm{A}_{\rm{i}}^*(\bm{I}_m-\bm{U}_{\rm{s}}\bm{U}_{\rm{s}}^*)\nonumber \\
        &-\bm{V}_{\rm{s}}\bm{\Sigma}_{\rm{s}}^{-1}\bm{U}_{\rm{s}}^*\bm{A}_{\rm{i}}\bm{V}_{\rm{s}}\bm{\Sigma}_{\rm{s}}^{-1}\bm{U}_{\rm{s}}^*\Big]\epsilon\;.\label{33}
    \end{align}
\end{enumerate}
\end{theorem}
\begin{proof}
The equivalence of (i) and (ii) comes from \Cref{the3.1summ}.

(ii)$\Rightarrow$(iii): Suppose that the CDSVD of $\bm{A}$ exists and there is a particular pair $(\bm{U}_{\rm{s}},\bm{V}_{\rm{s}})$ such that $\bm{A} = \bm{U\Sigma V}^* = \bm{U}_{\rm{s}}\bm{\Sigma}_{\rm{s}}\bm{V}_{\rm{s}}^*+(\bm{U}_{\rm{i}}\bm{\Sigma}_{\rm{s}} \bm{V}_{\rm{s}}^*+\bm{U}_{\rm{s}}\bm{\Sigma}_{\rm{i}} \bm{V}_{\rm{s}}^*+\bm{U}_{\rm{s}}\bm{\Sigma}_{\rm{s}} \bm{V}_{\rm{i}}^*)\epsilon$ is a CDSVD of $\bm{A}$. 
Then the DMPGI of $\bm{A}$ exists and has the form $\bm{A}^{\dag} = \bm{V\Sigma^{-1} U}^*$, which exactly satisfies \Cref{defi:DMPGI}.
By applying \Cref{pro1.1} $({\rm{iv}})$, we get $\bm{\Sigma}^{-1} = \bm{\Sigma}_{\rm{s}}^{-1} -\bm{\Sigma}_{\rm{s}}^{-1}\bm{\Sigma}_{\rm{i}}\bm{\Sigma}_{\rm{s}}^{-1}\epsilon $. Then it yields that
\begin{align}
    \bm{A}^{\dag} &= \bm{V\Sigma}^{-1}\bm{U}^*\nonumber\\
    &= (\bm{V}_{\rm{s}}+\bm{V}_{\rm{i}}\epsilon)\left(\bm{\Sigma}_{\rm{s}}^{-1} -\bm{\Sigma}_{\rm{s}}^{-1}\bm{\Sigma}_{\rm{i}}\bm{\Sigma}_{\rm{s}}^{-1}\epsilon\right)(\bm{U}_{\rm{s}}^*+\bm{U}_{\rm{i}}^*\epsilon)\nonumber\\
    &= \bm{V}_{\rm{s}}\bm{\Sigma}_{\rm{s}}^{-1}\bm{U}_{\rm{s}}^*+\left(\bm{V}_{\rm{i}}\bm{\Sigma}_{\rm{s}}^{-1}\bm{U}_{\rm{s}}^*+\bm{V}_{\rm{s}}\bm{\Sigma}_{\rm{s}}^{-1}\bm{U}_{\rm{i}}^*-\bm{V}_{\rm{s}}\bm{\Sigma}_{\rm{s}}^{-1}\bm{\Sigma}_{\rm{i}}\bm{\Sigma}_{\rm{s}}^{-1}\bm{U}_{\rm{s}}^*\right)\epsilon\;.\label{34}\nonumber
\end{align}
Noting the expressions of $\bm{U}_{\rm{i}}, \bm{V}_{\rm{i}}, \bm{\Sigma}_{\rm{i}}$, given in \Cref{le2.2} \cref{11a,11b,11c}, we obtain
\begin{subequations}
\begin{align*}
\bm{V}_{\rm{i}}\bm{\Sigma}_{\rm{s}}^{-1}\bm{U}_{\rm{s}}^* &= \bm{V}_{\rm{s}}\bm{Q}\bm{\Sigma}_{\rm{s}}^{-1}\bm{U}_{\rm{s}}^*+(\bm{I}_n-\bm{V}_{\rm{s}}\bm{V}_{\rm{s}}^*)\bm{A}_{\rm{i}}^*\bm{U}_{\rm{s}}\bm{\Sigma}_{\rm{s}}^{-2}\bm{U}_{\rm{s}}^*\label{35a}\;,\\
\bm{V}_{\rm{s}}\bm{\Sigma}_{\rm{s}}^{-1}\bm{U}_{\rm{i}}^* &= \bm{V}_{\rm{s}}\bm{\Sigma}_{\rm{s}}^{-1}\bm{P}^*\bm{U}_{\rm{s}}^* + \bm{V}_{\rm{s}}\bm{\Sigma}_{\rm{s}}^{-2}\bm{V}_{\rm{s}}^*\bm{A}_{\rm{i}}^*(\bm{I}_m-\bm{U}_{\rm{s}}\bm{U}_{\rm{s}}^*)\;,\\
\bm{V}_{\rm{s}}\bm{\Sigma}_{\rm{s}}^{-1}\bm{\Sigma}_{\rm{i}}\bm{\Sigma}_{\rm{s}}^{-1}\bm{U}_{\rm{s}}^* &= \bm{V}_{\rm{s}}\bm{\Sigma}_{\rm{s}}^{-1}\bm{U}_{\rm{s}}^*\bm{A}_{\rm{i}}\bm{V}_{\rm{s}}\bm{\Sigma}_{\rm{s}}^{-1}\bm{U}_{\rm{s}}^* - \bm{V}_{\rm{s}}\bm{\Sigma}_{\rm{s}}^{-1} \bm{P}\bm{U}_{\rm{s}}^*-\bm{V}_{\rm{s}}\bm{Q}^*\bm{\Sigma}_{\rm{s}}^{-1}\bm{U}_{\rm{s}}^*\;.
\end{align*}
\end{subequations}
Besides, 
\begin{align}
&\bm{V}_{\rm{s}}\bm{Q}\bm{\Sigma}_{\rm{s}}^{-1}\bm{U}_{\rm{s}}^*+\bm{V}_{\rm{s}}\bm{\Sigma}_{\rm{s}}^{-1}\bm{P}^*\bm{U}_{\rm{s}}^*+\bm{V}_{\rm{s}}\bm{\Sigma}_{\rm{s}}^{-1} \bm{P}\bm{U}_{\rm{s}}^*+ \bm{V}_{\rm{s}}\bm{Q}^*\bm{\Sigma}_{\rm{s}}^{-1}\bm{U}_{\rm{s}}^*\nonumber\\
    =&
    \bm{V}_{\rm{s}}\Big[ (\bm{Q}+\bm{Q}^*)\bm{\Sigma}_{\rm{s}}^{-1}+\bm{\Sigma}_{\rm{s}}^{-1}(\bm{P}^*+\bm{P}) \Big]\bm{U}_{\rm{s}}^*
    =\bm{O}_{n\times m}\;,\nonumber
\end{align}
where the skew-Hermitian property of $\bm{P}$ and $\bm{Q}$ in \Cref{lemma3.1} is applied. 
Thus, $\bm{A}^\dag$ has the expression in \cref{33} by using this particular pair $(\bm{U}_{\rm{s}},\bm{V}_{\rm{s}})$.

Next, we aim to prove that \cref{33} holds for any pair. Assume that $\bm{A}_{\rm{s}} = \bm{M}_{\rm{s}}\bm{\Sigma}_{\rm{s}}\bm{N}_{\rm{s}}^*$ is another compact SVD of $\bm{A}_{\rm{s}}$. 
Then from the proof of \Cref{the3.1summ}, there exists a block diagonal unitary matrix $\bm{X}$ partitioned as in $\bm{\Sigma}_{\rm{s}}$ such that $\bm{U}_{\rm{s}}=\bm{M}_{\rm{s}}\bm{X}$ and $\bm{V}_{\rm{s}}=\bm{N}_{\rm{s}}\bm{X}$. 
Resulting from the fact that $\bm{X}\bm{\Sigma}_{\rm{s}}^{-1}\bm{X}^* = \bm{\Sigma}_{\rm{s}}^{-1}$ and $\bm{X}\bm{\Sigma}_{\rm{s}}^{-2}\bm{X}^* = \bm{\Sigma}_{\rm{s}}^{-2}$, \cref{33} is satisfied for the pair $(\bm{M}_{\rm{s}},\bm{N}_{\rm{s}})$.

Hence, the condition (iii) holds and the expression of $\bm{A}^{\dag}$ is obtained.

(iii)$\Rightarrow$(i): Suppose that the DMPGI of $\bm{A}$ exists and $\bm{A}^\dag$ has the form expressed in \cref{33}.
Then we have $\bm{A}\bm{A}^\dag\bm{A}=\bm{A}$ by the second property of the DMPGI. In other words, the infinitesimal parts of both sides are equal. Simplification yields $  \bm{A}_{\rm{i}}\bm{V}_{\rm{s}}\bm{V}_{\rm{s}}^*     +\bm{U}_{\rm{s}}\bm{U}_{\rm{s}}^*\bm{A}_{\rm{i}}      -\bm{U}_{\rm{s}}\bm{U}_{\rm{s}}^*\bm{A}_{\rm{i}}\bm{V}_{\rm{s}} \bm{V}_{\rm{s}}^*=\bm{A}_{\rm{i}} $. Thus, the condition (i) holds.

Consequently, we conclude that conditions (i), (ii) and (iii) are equivalent.
\end{proof}

Indeed, our proposed necessary and sufficient conditions for a dual complex matrix to have the DMPGI and its explicit expressions are precisely accordant with those developed by Wang \cite{wang2021characterizations}.

\section{Numerical Experiments}\label{sec:numerical}
In this section, numerical comparisons with some existing algorithms are reported and the role of infinitesimal parts is demonstrated in simulated time-series data by taking full advantage of the relationship between standard and infinitesimal parts. 
Moreover, traveling wave identification in a small-scale road monitoring video and large-scale brain fMRI data verifies the applicability of our proposed CDSVD. We wrote our codes in Matlab R2020a and conducted all experiments on a personal computer with Intel(R) Core(TM) i7-10510U CPU @ 1.80GHz 2.30 GHz and 16G memory.

\subsection{Comparisons of Algorithms for Dual SVD}

We start by comparing the CDSVD algorithm with other SVD algorithms of dual matrices, namely the dual SVD of dual real full column rank matrices (DSVD) \cite{pennestri2018moore} and the SVD of dual complex matrices (SVD-QACLL) \cite{qi2022low}.

Afterward, we study how well the SVD-QACLL, DSVD and our proposed CDSVD algorithm perform in simulations. Dual real matrices and dual complex matrices are generated by {\texttt{randn}}, which are always with full rank. Dual matrices of 8 different sizes are chosen to calculate the SVD and the running time is the evaluation index of algorithm performance. \Cref{table1} shows the average running time of SVD of dual real matrices and dual complex matrices with different sizes by using the three algorithms across the simulations repeated 50 times, and values in parentheses represent standard deviations. In addition, ``-{}-" stands for out of memory or taking over 1,000 seconds. The results indicate that our CDSVD \Cref{alg:CDSVD} compares favorably to the other two algorithms for dual real matrices, and is also distinctly better than the SVD-QACLL for dual complex matrices, especially for large sizes.
\begin{table}[htpb]
	\centering
	\caption{Comparisons of three algorithms for the SVD of dual matrices}
	\label{table1}
	\renewcommand{\arraystretch}{1.5}
	\resizebox{\textwidth}{!}{
	\begin{tabular}{|c|c|c|c|c|c|}
	\hline
	Matrix Types & \multicolumn{3}{c|}{Dual Real Matrices} & \multicolumn{2}{c|}{Dual Complex Matrices}\\
	\hline
	\diagbox{Size}{Time (s)}{Alg.}  & CDSVD & SVD-QACLL & DSVD & CDSVD & SVD-QACLL \\
	\hline
    $20\times 10$ & $0.0003\; (0.0001)$ & $0.0010\; (0.0004)$ & $0.0349\; ( 0.0096)$ & $0.0004\; ( 0.0002)$ & $0.0026\; ( 0.0107)$\\
    \hline
    $50\times 25$ & $0.0013\; ( 0.0003)$ & $0.0053\; ( 0.0011)$ & $20.9404\; ( 0.0879)$ & $0.0010\; ( 0.0003$ & $0.0038\; ( 0.0074)$\\
    \hline
    $100\times 50$ & $0.0011\; ( 0.0003)$ & $0.0043\; ( 0.0010)$ & -{}- & $0.0024\; ( 0.0010)$ & $0.0060\; ( 0.0012)$\\
    \hline
    $200\times100$ & $0.0029\; ( 0.0004)$ & $0.0115\; ( 0.0011)$ & -{}- & $0.0089\; ( 0.0038)$ & $0.0153\; ( 0.0016)$\\
    \hline
    $1000\times500$ & $0.1339\; ( 0.0082)$ & $0.7127\; ( 0.0461)$ & -{}- & $0.6342\; ( 0.0243)$ & $1.7926\; ( 0.0311)$\\
    \hline
    $2000\times1000$ & $0.8720\; ( 0.0490)$ & $4.8298\; ( 0.1242)$ & -{}- & $5.7022\; ( 0.3600)$ & $10.9719\; ( 0.3794)$\\
    \hline
    $5000\times2500$ & $18.5958\; ( 3.0477)$ & $96.5999\; ( 7.8840)$ & -{}- & $
  110.0527 \; (  11.1930)$ & $195.1239 \; (   16.6747)$\\
    \hline
    $10000\times5000$ & $157.9064\; ( 38.9427)$ & $  708.4614\; ( 37.9841)$ & -{}- & $866.9033 \; ( 50.8285)$ & -{}-\\
    \hline
	\end{tabular}}
\end{table}

\subsection{Comparison Between SVD and CDSVD}
Classical SVD decomposes a high-dimensional matrix into the sum of some uncorrelated rank-1 matrices. 
Each singular value reflects the 
weight of the corresponding rank-1 matrix in the original matrix.
When Gaussian noise is added to a matrix of initial rank $r$, the matrix becomes full rank, making it difficult to determine the true rank by SVD.
Here, we attempt to utilize the CDSVD to obtain more information and then explore the true rank of the original matrix from the noise-added matrix.

Consider a spatiotemporal propagation pattern \cite{feeny2008complex}
\begin{align}
    \bm{x}(t) = 2e^{\gamma t}\big[\cos{(\omega t)}\bm{c}-\sin{(\omega t)}\bm{d}\big]\;,\label{travelingwave}
\end{align}
where $\bm{x}(t)$ represents a real vector of particle positions at time $t$, $\gamma,\omega$ are real scalars representing the exponential decay rate and angular frequency, and $\bm{c},\bm{d}$ are real vectors representing two modes of oscillation. 
This reflects the circularity and continuity of propagation between $\bm{c}$ and $\bm{d}$.
In addition, \cref{travelingwave} behaves as a ``standing wave" of rank-1 when $\bm{c}=\bm{d}$ because only oscillation with up and down can be observed. It also behaves as a ``traveling wave" of rank-2 when $\bm{c}\neq\bm{d}$ due to its appearance of a wavy motion.

In the following simulation, we aim to study whether the infinitesimal part of the CDSVD provides extra information compared to the classical SVD.
Construct a spatiotemporal matrix with each row representing the pattern of a point in space over time. \Cref{fig:svd_cdsvd_comparison} illustrates a rank-6 initial matrix $\bm{A}$ consisting of four standing waves and one traveling wave, and the green circles represent the singular values of $\bm{A}$.
A new matrix $\widetilde{\bm{A}}$ emerges when Gaussian noise is added to the original matrix $\bm{A}$ and the signal-to-noise ratio (SNR) is 0.1604.
The red squares show the singular values of the noise-added matrix $\widetilde{\bm{A}}$.
Except for the first four obvious singular values, it is difficult to judge whether the remaining singular values are inherited from the original matrix or promoted by noise.
Thus, the SVD of the noise-added matrix $\widetilde{\bm{A}}$ can not recognize the true rank of the original matrix $\bm{A}$.
\begin{figure}[htpb]
    \centering \includegraphics[scale=0.48]{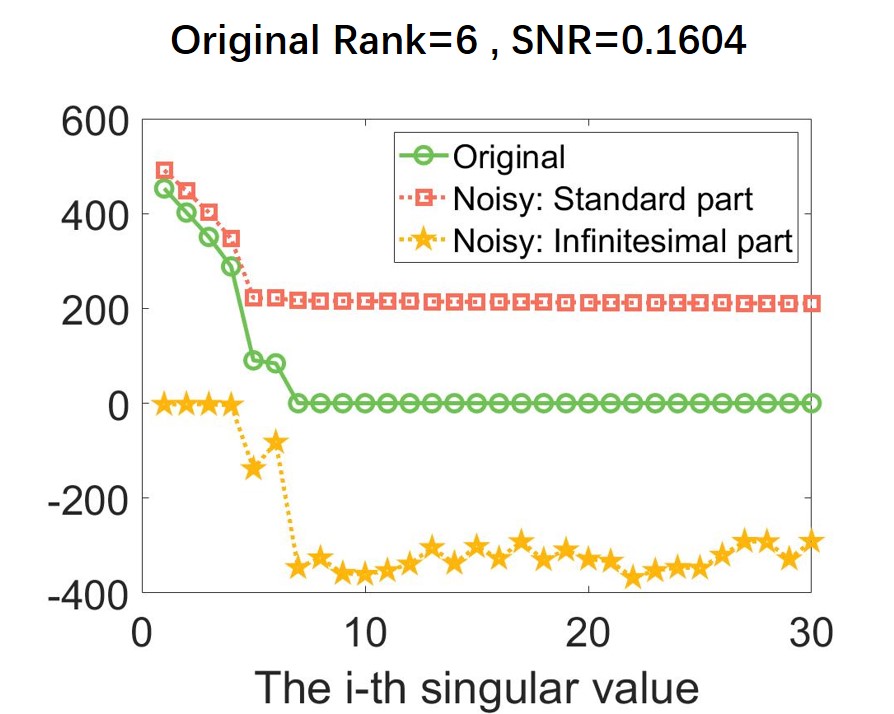}
    \caption{Comparison of singular values after the addition of Gaussian noise.}
    \label{fig:svd_cdsvd_comparison}
\end{figure}

In view of this, we create a dual matrix $\bm{X}$, whose standard part is $\widetilde{\bm{A}}$ and infinitesimal part is the first-order difference of $\widetilde{\bm{A}}$. 
The standard part of each singular value of the dual matrix $\bm{X}$, which is also a singular value of $\widetilde{\bm{A}}$, is once again shown by a red square.
Additionally, the yellow pentagrams refer to the infinitesimal part of singular values of the dual matrix $\bm{X}$.
It is distinctly noticed from yellow pentagrams that the first six singular values are significantly larger than others, which separates the original valid data from the added noise.
That is, we identify the true rank of the original matrix $\bm{A}$ from the infinitesimal part of singular values of the dual matrix $\bm{X}$ based on the CDSVD.

This is where the advantages of the CDSVD come into play, i.e., for a low-rank time-series matrix, it is hard to know the original rank after adding noise by using the classical SVD. However, if we construct a corresponding dual matrix, then its CDSVD provides an approach to identify the original rank by observing the infinitesimal part of singular values.

\subsection{Simulations of Standing and Traveling Waves}
The implication of the infinitesimal part in practice is a focus of our research.
Eduard Study \cite{study1923grenzfall} proposed a dual angle between two skew lines in three-dimensional space, whose standard part and infinitesimal part represent angle and distance respectively. 
Whereas, we regard original time-series data as the standard part and its derivative or first-order difference as the infinitesimal part, which naturally identifies standing and traveling waves in the following simulations.

With the spatiotemporal propagation pattern \cref{travelingwave}, we continue with further discussion of standing and traveling waves.

To gain insight into their properties, we begin with developing Gaussian waves on a two-dimensional grid constructed as a 50-by-100 array of pixels. 
When $m$ particles are arranged in rows and $n$ time points in columns, an $m$-by-$n$ ensemble matrix $\widehat{\bm{X}}$ is built, and $\bm{c}$, $\bm{d}$ in \cref{travelingwave} are taken as vectors obtained by straightening two-dimensional Gaussian curves. 
Then a dual matrix $\bm{X}$ is developed, whose standard part is $\widehat{\bm{X}}$ and infinitesimal part is the derivative of $\widehat{\bm{X}}$ respect to time. 
When $\bm{c} = \bm{d}$, \cref{travelingwave} behaves as a standing wave.
\Cref{fig: simulation wave}a shows the standard and the infinitesimal parts of $\bm{U}$ after calculating the CDSVD of $\bm{X}$. 
The central position of the standing wave is quite clearly observed at (25,50) in $\bm{U}_{\rm{s}}$. 
Conversely, if $\bm{c} \neq \bm{d}$, \cref{travelingwave}  is a traveling wave.
\Cref{fig: simulation wave}b presents the first two columns of $\bm{U}$, and the two central positions of the traveling wave are displayed at (25,20) and (25,80) in $\bm{U}_{\rm{s}}$. It is worth mentioning that the values of $\bm{U}_{\rm{i}}$ of traveling waves are relatively large compared to that of standing waves.

\begin{figure}[htpb]
    \centering
    \includegraphics[scale = 0.34]{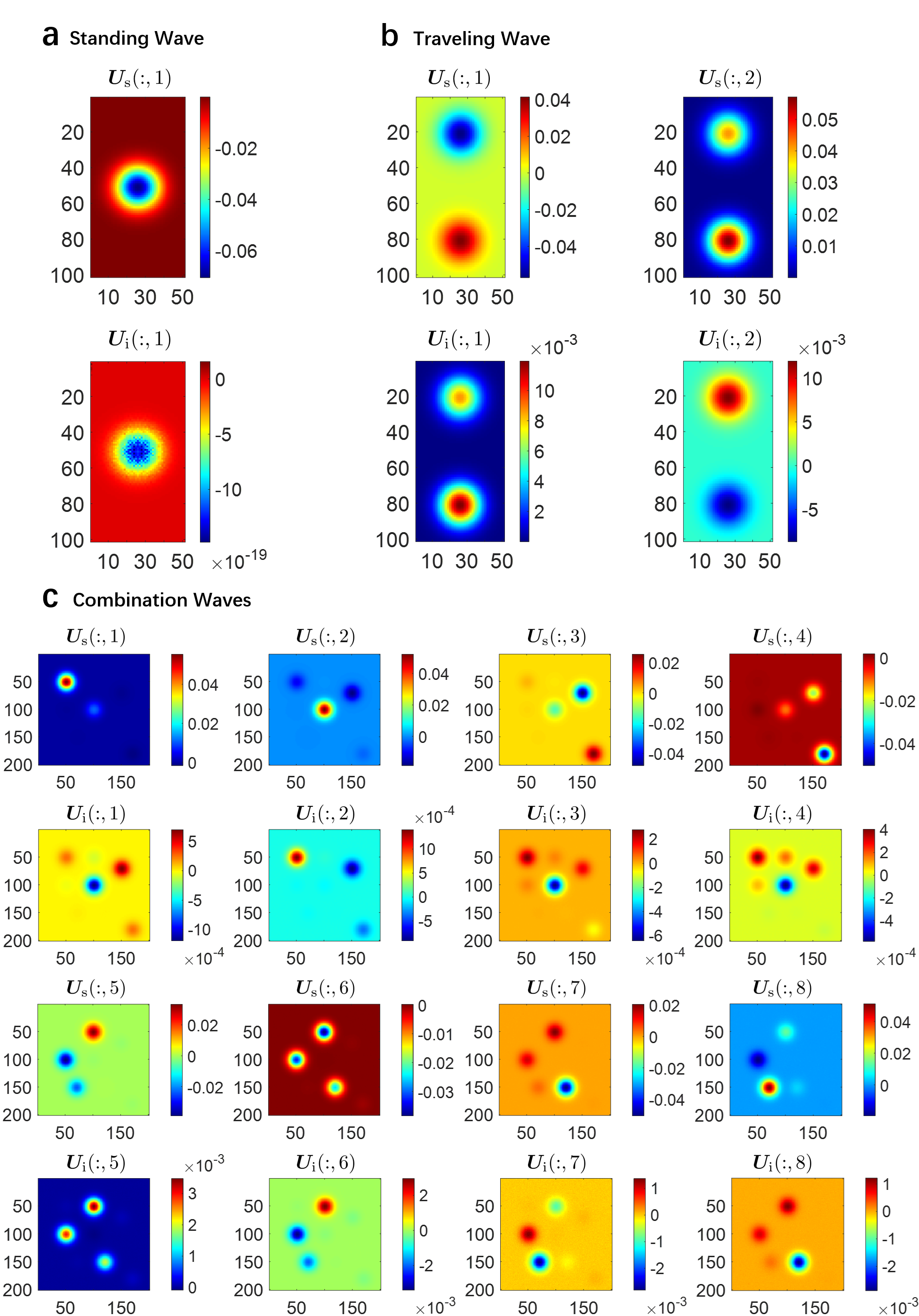}
    \caption{Simulation of a pure standing wave presents its central position in the standard part of $\bm{U}$ and small values in the infinitesimal part. Simulation of a pure traveling wave reveals the pairs of similarities between standard parts and infinitesimal parts of $\bm{U}$. The component waves are accurately identified in combination waves consisting of four standing waves and two traveling waves based on the CDSVD and special properties of traveling waves.}
    \label{fig: simulation wave}
\end{figure}

Surprisingly, we observe there exist pairs of similarities between $\bm{U}_{\rm{s}}$ and $\bm{U}_{\rm{i}}$ for the traveling waves. For instance, in \Cref{fig: simulation wave}b, $\bm{U}_{\rm{s}}(:,2)$ and $\bm{U}_{\rm{i}}(:,1)$ have two points located in same positions with a positive multiple, compared to $\bm{U}_{\rm{s}}(:,1)$ and $\bm{U}_{\rm{i}}(:,2)$ having two points with same positions but a negative multiple. 
It is worth considering whether traveling waves always have such properties as described above.

The following proposition substantiates that a rank-2 dual real matrix with special structures has pairs of similarities between its singular vectors, where one is homologous and the other is heterologous.

\begin{proposition}
\label{pro6.1}
Suppose $\bm{A} = \bm{A}_{\rm{s}}+\bm{A}_{\rm{i}}\epsilon\in\mathbb{DR}^{m\times n}$ is a rank-2 dual real matrix with column space ${\rm{Ran}}(\bm{A}_{\rm{i}})\subseteq {\rm{Ran}}(\bm{A}_{\rm{s}})$. Let $\bm{A} = \bm{U\Sigma V}^T$ be its CDSVD with $\bm{U} = \bm{U}_{\rm{s}}+\bm{U}_{\rm{i}}\epsilon\in\mathbb{DR}^{m\times 2}$. Then the columns of $\bm{U}$ satisfy
\begin{subequations}
    \begin{align}
        \bm{U}_{\rm{i}}(:,1) &= \alpha \bm{U}_{\rm{s}}(:,2)\;,\\
       \bm{U}_{\rm{i}}(:,2) &= \beta \bm{U}_{\rm{s}}(:,1)\;,\\
       \alpha\beta &< 0\;.
    \end{align}
\end{subequations}
\end{proposition}

\begin{proof}
According to the CDSVD, $\bm{U}_{\rm{i}}$ can be expressed as follows based on \cref{5555555a,11a}:
\begin{align}
    \left\{\begin{array}{l}
        \bm{U}_{\rm{i}}=\bm{U}_{\rm{s}}\bm{P}+(\bm{I}_m-\bm{U}_{\rm{s}}\bm{U}_{\rm{s}}^T)\bm{A}_{\rm{i}}\bm{V}_{\rm{s}}\bm{\Sigma}_{\rm{s}}^{-1}\;,\\
        \bm{P}+\bm{P}^T=\bm{O}\;,\bm{P}\neq\bm{O}\;.
    \end{array}\right.\nonumber
\end{align}
Suppose that $\bm{U}_{\rm{s}} = [\bm{u}_1\;,\bm{u}_2],\;\bm{P} = (p_{ij})\in\mathbb{R}^{2\times2}$, then $p_{11} = p_{22} = 0$, $p_{12}+p_{21}=0$ and $p_{12},p_{21}\neq 0$ due to the skew-Hermitian and nonzero $\bm{P}$. 
In addition, there exists $\bm{M}\in\mathbb{R}^{n\times n}$ such that $\bm{A}_{\rm{i}}=\bm{A}_{\rm{s}}\bm{M}=\bm{U}_{\rm{s}}\bm{\Sigma}_{\rm{s}}\bm{V}_{\rm{s}}^T\bm{M}$ from the properties of $\bm{A}_{\rm{i}}$. Hence, we have
\begin{align}
\bm{U}_{\rm{i}}&=\bm{U}_{\rm{s}}\bm{P}+(\bm{I}_m-\bm{U}_{\rm{s}}\bm{U}_{\rm{s}}^T)\bm{A}_{\rm{i}}\bm{V}_{\rm{s}}\bm{\Sigma}_{\rm{s}}^{-1}\nonumber\\
    &=\left[\begin{array}{cc}
       \bm{u}_1  & \bm{u}_2  
    \end{array}\right]\left[\begin{array}{cc}
       0  & p_{12}\\
       p_{21} & 0
    \end{array}\right]+(\bm{I}_m-\bm{U}_{\rm{s}}\bm{U}_{\rm{s}}^T)\bm{U}_{\rm{s}}\bm{\Sigma}_{\rm{s}}\bm{V}_{\rm{s}}^T\bm{M}\bm{V}_{\rm{s}}\bm{\Sigma}_{\rm{s}}^{-1}\nonumber\\
    &= \left[\begin{array}{cc}
       p_{21}\bm{u}_2  & p_{12}\bm{u}_1  
    \end{array}\right]\nonumber\;,
\end{align}
which implies that $\bm{U}_{\rm{i}}(:,1) = p_{21}\bm{U}_{\rm{s}}(:,2)\;,\;\bm{U}_{\rm{i}}(:,2)= p_{12}\bm{U}_{\rm{s}}(:,1)$. In addition, $p_{21}p_{12}<0$. This completes the proof.
\end{proof}

The above proposition reveals pairs of similarities in a traveling wave because its corresponding dual matrix is real rank-2 and the column space of its infinitesimal part is contained in that of its standard part. This achieves the unification of theory and practice. Furthermore, we discuss whether the property can be reserved and used for identifying traveling waves in combination waves. 

To further substantiate the effectiveness of \Cref{pro6.1} in traveling wave identification, we construct a combination wave and aim to separate its component waves. Here, four standing waves and two traveling waves are combined together with the same standard deviation $\sigma = 1$ but different angular frequencies and weights. Gauss noise with peak value $10^{-3}$ is also added. 
Thus, a combination matrix is generated as the standard part of a dual matrix $\bm{Y}$, and its first-order difference in relation to time forms the infinitesimal part of $\bm{Y}$. After computing the CDSVD of $\bm{Y}$, $\bm{U}$ is used to identify the six different waves. As we theorized above, $\bm{Y}_{\rm{s}}$ has 8 nonzero singular values. \Cref{fig: simulation wave}c illustrates the corresponding maps of 8 columns of $\bm{U}_{\rm{s}}$ and $\bm{U}_{\rm{i}}$. The first four columns of $\bm{U}$ identify the relevant four most weighted standing waves. The point with the largest absolute value in every map of $\bm{U}_{\rm{s}}$ is the central position of the standing wave, and the value of corresponding $\bm{U}_{\rm{i}}$ is relatively small. 
Hence, (50,50), (100,100), (150,70), and (170,180) are the peaks of the four standing waves. The fifth and the sixth columns of $\bm{U}$ satisfy the similarity property in \Cref{pro6.1} intuitively, which reveals that these two form a traveling wave. (50,100) and (100,50) are the peaks of this traveling wave, since these two match the largest absolute values in the maps of the fifth and sixth vectors of $\bm{U}_{\rm{s}}$. Similarly, (120,150) and (70,150) are the peaks of the other traveling wave indicated by the seventh and eighth vectors of $\bm{U}_{\rm{s}}$.

As a consequence, for the first time, we propose a method of identifying traveling waves and extracting their propagation paths. To the best of our knowledge, existing methods \cite{bolt2022brainwave,feeny2008complex} only defined an index to measure what percentage of a wave is a traveling wave while being unable to separate the exact traveling wave from the original one.

\subsection{Applications to Time-Series Analysis}

We have proffered how to use pairs of similarities between columns of $\bm{U}$ under the CDSVD in \Cref{pro6.1} to extract traveling waves in simulated experiments. Here, we aim to verify the applicability of the above method and further validate the role of infinitesimal parts in a classical visual foreground-background separation problem.  

A road monitoring video showing the movement of cars and pedestrians at the intersection is taken as an example, whose resolution is 320×240 pixels \cite{intersectionVideo}.

\begin{figure}
    \centering
    \includegraphics[scale = 0.39]{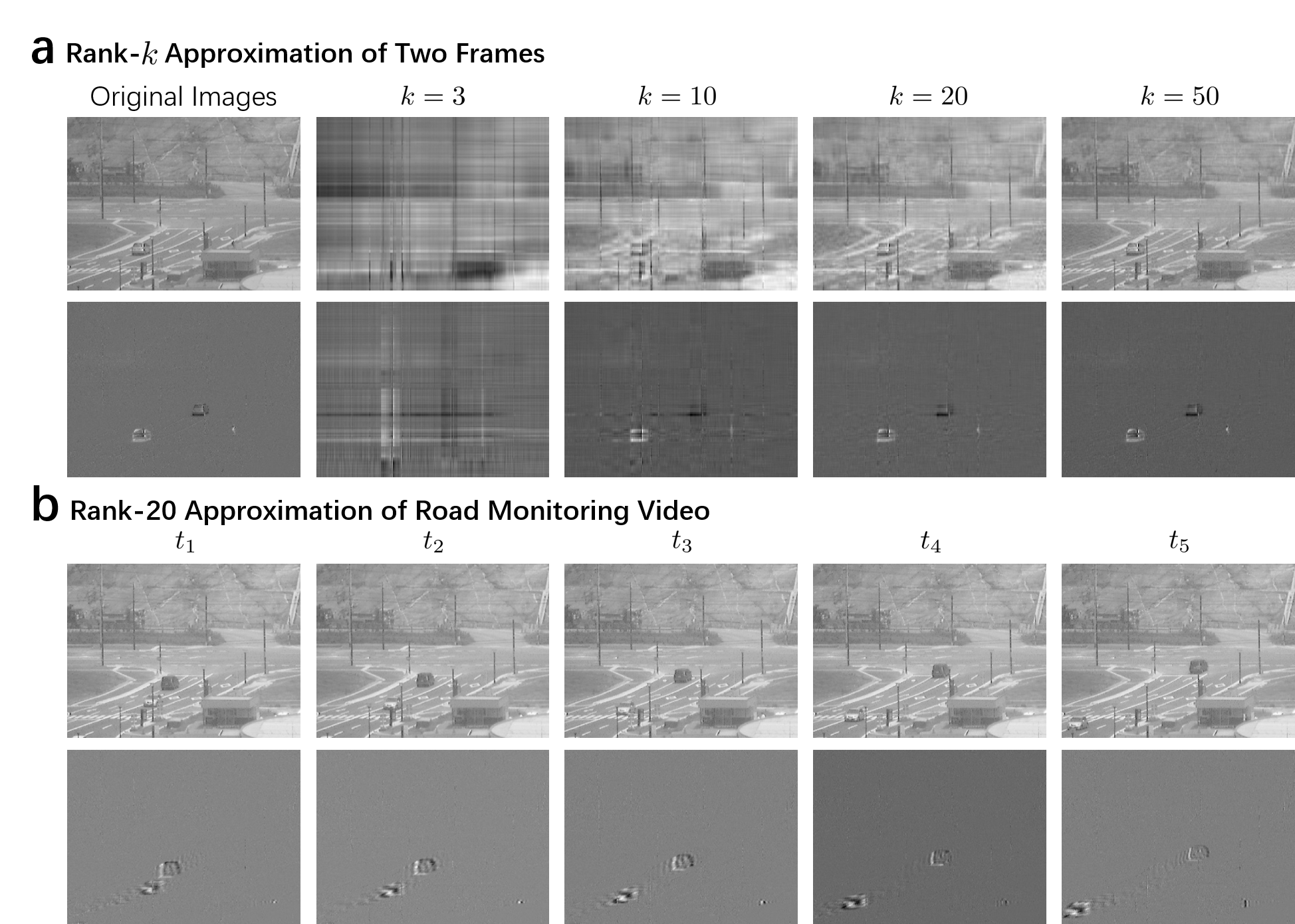}
    \caption{Rank-$k$ approximation of two frames under the quasi-metric directly illustrates an origin and an end of the car's motion. Rank-20 approximation of the road monitoring video under the quasi-metric separates the moving foreground from the fixed background and further describes the process of two cars going away from each other.}
    \label{fig:F-B separation}
\end{figure}

First, two photos at different moments are extracted from the video and our goal is to separate the still background and the moving foreground and to witness the outcome of movements during this time. A dual real matrix $\bm{A}$ is constructed, whose standard part, $\bm{A}_{\rm{s}}$, is the gray-scale matrix of the photo at the previous moment and infinitesimal part, $\bm{A}_{\rm{i}}$, is the differences of gray-scale matrices between the two moments. Thus, we use \Cref{alg:lowrank} to compute the optimal rank-$k$ approximate dual real matrices of $\bm{A}$, $\bm{A}_k$, under the quasi-metric. The first column in \Cref{fig:F-B separation}a presents $\bm{A}_{\rm{s}}$ and $\bm{A}_{\rm{i}}$ separately. Other columns illustrate the results of the low rank approximation with different rank-$k$, where the first row shows the standard parts of $\bm{A}_k$ and the second row shows their corresponding infinitesimal parts reporting the result of movements during this time. 
When $k=50$, the infinitesimal part of $\bm{A}_k$ shows the moving foreground, in which there is a white car and a black car representing the initial and the final positions respectively. 
In addition, the small white point stands for the initial position of the pedestrian but the absence of a small black point means that the person is out of sight at the last moment.

Furthermore, we aim to extract a video of the road condition information that contains only moving vehicles and pedestrians. We first capture 5,000 consecutive frames of photos in the original video, and the gray-scale matrix of each frame is then straightened into a column vector, all the 5000 vectors are concatenated into a matrix $\widehat{\bm{B}}$ with size 76,800-by-5,000. Next, we use the matrix $\widehat{\bm{B}}$ to construct a dual matrix $\bm{B}$, whose standard part $\bm{B}_{\rm{s}}$ is its first 4998 columns and infinitesimal part $\bm{B}_{\rm{i}}$ is its one-sided, second-order approximation of the first derivative respect to time by applying \cite{levy}
\begin{align}
    f^{'}(x) \approx \frac{-3f(x)+4f(x+h)-f(x+2h)}{2h}.\label{6.3}
\end{align}

After calculating the rank-$k$ approximation of $\bm{B}$ under the quasi-norm by using \Cref{alg:lowrank}, we obtain the rank-$k$ approximate dual matrix $\bm{B}_k$, and each column of its infinitesimal part is then reshaped into a matrix of the same size as the original gray-scale matrix. Accordingly, only moving vehicles and pedestrians can be contained when the 4,998 restored images were strung together into a video. The first row in \Cref{fig:F-B separation}b presents some frames of this video at adjacent 6 moments, and the second row manifests the corresponding infinitesimal parts of rank-20 approximate dual matrices. The trajectories of two cars facing away from each other can be clearly seen separately with no cluttered background.

\begin{figure}
    \centering
    \includegraphics[scale = 0.35]{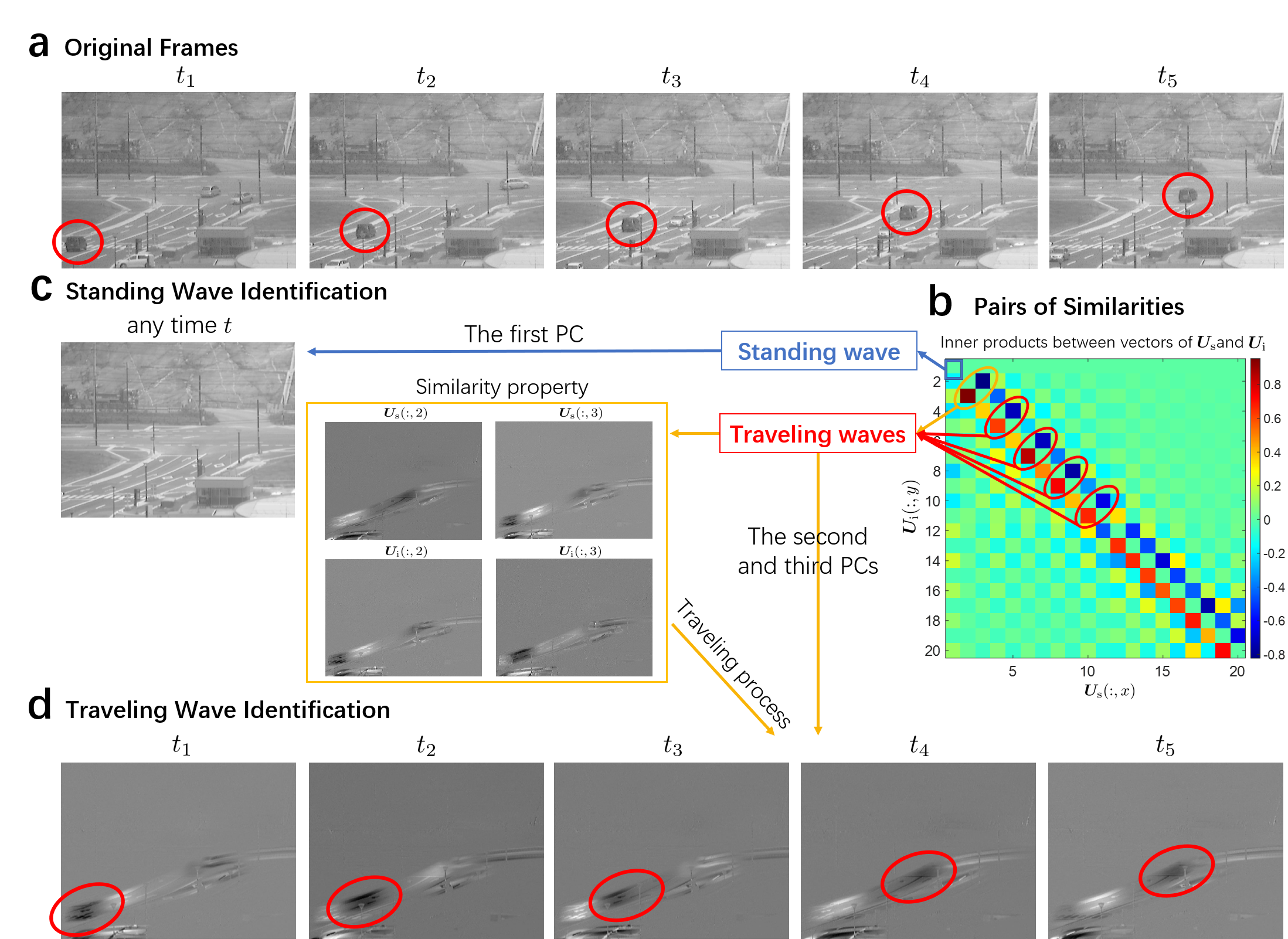}
    \caption{Pairs of similarities are illustrated between vectors of $\bm{U}_{\rm{s}}$ and $\bm{U}_{\rm{i}}$ in intersection data after calculating the CDSVD. A standing wave is identified based on the first PC and a video with only fixed background is extracted. The second and third PCs identify a traveling wave showing the trajectory of a black car corresponding to that presented in original frames.}
    \label{fig:car_traveling}
\end{figure}

By now, we have revealed the important role of infinitesimal parts in the low rank approximation, which extract only moving foreground from a fixed intersection. Additionally, we aim to explore whether there exist traveling waves in the intersection video based on \Cref{pro6.1}. 
For convenience, a dual sub-matrix $\bm{C}$ of $\bm{B}$ containing its continuous 130 columns is employed to identify standing and traveling waves in the calculation of the CDSVD, $\bm{C} = \bm{U\Sigma V}^T$. \Cref{fig:car_traveling}a shows five frames of $\bm{C}_{\rm{s}}$ at successive moments. Naturally, inner products between vectors of $\bm{U}_{\rm{s}}$ and $\bm{U}_{\rm{i}}$ are calculated and displayed in \Cref{fig:car_traveling}b. 
It is evident that inner products between $\bm{U}_{\rm{s}}(:,1)$ and any $\bm{U}_{\rm{i}}(:,y)$ are close to zero, and so are $\bm{U}_{\rm{i}}(:,1)$ and any $\bm{U}_{\rm{s}}(:,x)$, which implies that the first principal component (PC) yields a standing wave and its standard part witnesses the propagation process with invariable still background (see \cref{fig:car_traveling}c). The largest inner product, 0.9536, is reached between $\bm{U}_{\rm{s}}(:,2)$ and $\bm{U}_{\rm{i}}(:,3)$, and the smallest, -0.8211, is between $\bm{U}_{\rm{s}}(:,3)$ and $\bm{U}_{\rm{i}}(:,2)$. Thus, the second and the third columns of $\bm{U}$ illustrate the similarity property. Moreover, the sum of these two PCs contributes to a traveling wave with the propagating process from $\bm{U}_{\rm{s}}(:,3)$ to $\bm{U}_{\rm{s}}(:,2)$ shown in \cref{fig:car_traveling}d, where the movement of the conspicuous black block corresponds to that of the black car in \cref{fig:car_traveling}a.

We shed some light on the method of identifying traveling waves in time-series data. Particularly for a road monitoring video, we distinguish which PC can recognize a standing wave and which two PCs can generate a traveling wave, and then use the similarity property (\Cref{pro6.1}) to derive the moving process in the calculation of the standard part of the sum of those two PCs. It is remarkable that we use only rank-2 matrices to describe the obvious propagation process of the traveling wave, which illustrates the important role of infinitesimal parts.


\subsection{Applications to Traveling Wave Identification in the Brain}\label{sec:brain}

We have created an innovative technique to deal with time-series data as well as to locate intrinsic traveling waves, as demonstrated in the prior simulations and the small-scale intersection experiments.

Thereafter, we focus on the Human Connectome Project (HCP) database \cite{HCP} including high-resolution 3T MR scans from young healthy adults (ages 22-35) utilizing task-state fMRI (tfMRI), one of the most important imaging modalities, in which subjects were asked to execute tasks that were intended to activate a variety of cortical and subcortical networks. Each tfMRI scan was split into two runs: one with right-to-left phase encoding and the other with left-to-right phase encoding (in-plane FOV (field of view) rotation was achieved by inverting both the RO (readout) and PE (phase encoding) gradient polarity) \cite{HCPmanual}. 

In order to confirm the correspondence between our theoretical study of traveling wave identification and the empirical response of functional brain regions, we select the typical language processing (semantic and phonological processing) task from all the seven kinds of tasks to explore whether traveling waves are recognized in the brain regions involved in language. This language processing task was developed by Binder et al. \cite{languageTask}. There are 4 blocks of a story task and 4 blocks of a math task in this experiment. The story task presents brief auditory stories (5-9 sentences) followed by a 2-alternative forced-choice question that asks about the topic of the story and the math task also presents trials aurally and requires subjects to complete addition and subtraction problems.

\begin{figure}
    \centering
    \includegraphics[scale = 0.392]{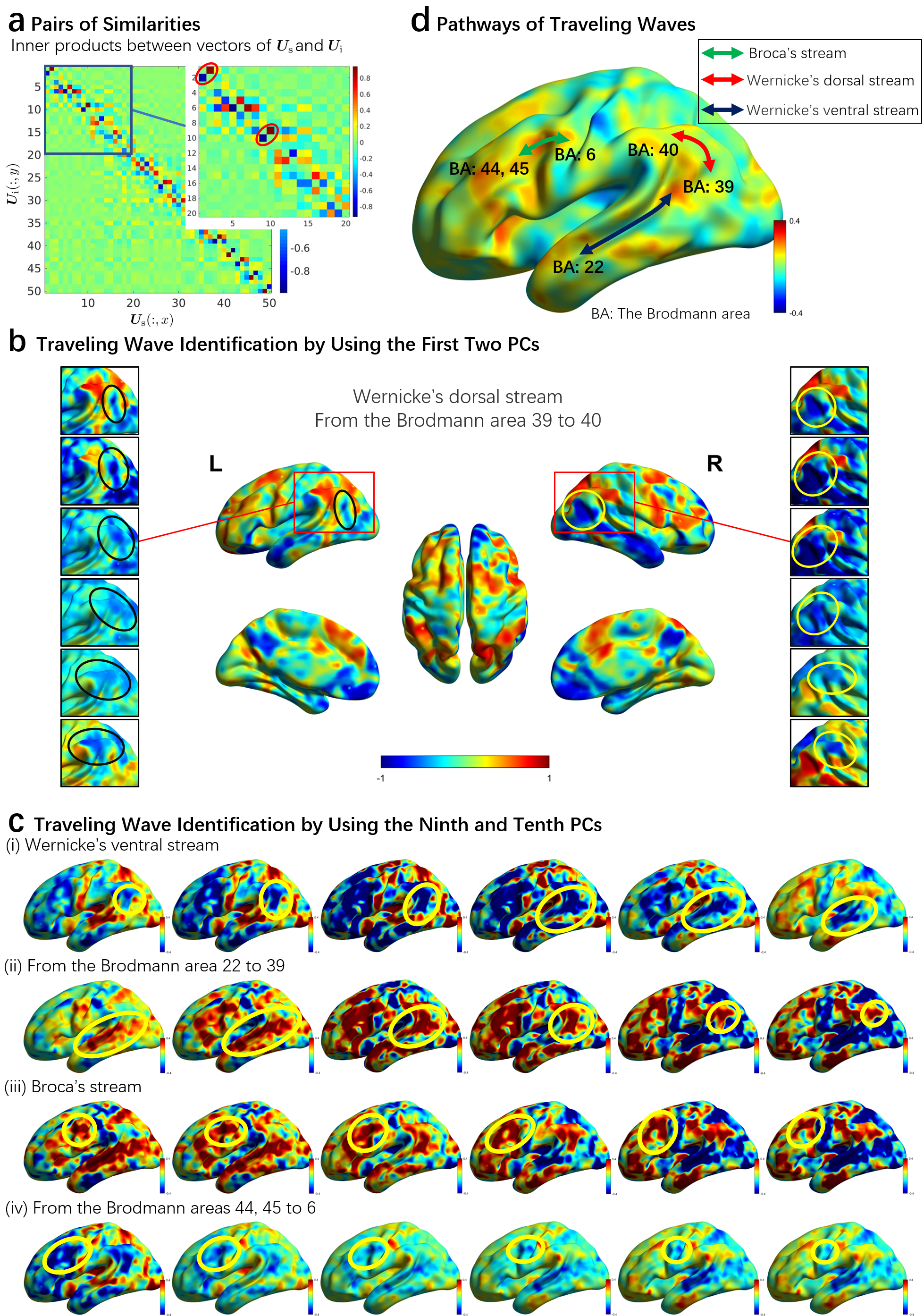}
    \caption{Pairs of similarities are explicated between vectors of $\bm{U}_{\rm{s}}$ and $\bm{U}_{\rm{i}}$ in the brain tfMRI data with language processing task after calculating the CDSVD. The first two PCs identify a stream from the Brodmann area 39 to 40, which are both parts of Wernicke's area. In addition, compared to the right stream, the left stream is more pronounced verifying that the left hemisphere is the main processing system of language. Then the ninth and tenth PCs identify two kinds of traveling waves involving Brodmann areas 22, 39, 44, 45 and 6. In summary, paths of traveling waves including Wernicke's dorsal and ventral streams as well as Broca's stream are described, which are all located in brain regions related to the language task.}
    \label{fig:brain_traveling}
\end{figure}

One subject is randomly selected in our experiment. In his corresponding tfMRI sequential data with language processing task, a 4-D voxel-based image including 91-by-109-by-91 3-D spatial position signals and 316 frames is provided, whose voxel size is 2-mm and repetition time $\text{TR}=0.72$ seconds. Data pre-processing is essential so that the image matrix can be lean but still contain a lot of information. The tfMRI signals were temporally filtered to the conventional low-frequency range using a Butterworth band-pass zero-phase filter (0.01–0.1 Hz) and were then spatially smoothed with a 5-mm full width at half max of 3-dimensional Gaussian lowpass filter using convolution and ignoring the NaNs. To unify magnitudes of the same voxel in different frames, we then Z-scored (to zero mean and unit variance) the BOLD time series from all voxels. 

In the second stage, we aim to construct a dual matrix taking advantage of the original brain sequential data matrix and then corroborate the consistency between brain regions where evident traveling waves are extracted and the corresponding cerebral cortex function. After straightening the first three spatial dimensions into a column vector for each frame, a real matrix is generated when the fourth time dimension is spliced sequentially, which is regarded as the standard part of a dual matrix $\bm{D}$. Similarly, the first-order difference of the standard part about time by applying \cref{6.3} is regarded as the infinitesimal part of $\bm{D}$. In our calculation of the CDSVD $\bm{D} = \bm{U \Sigma V}^T$ using \Cref{alg:CDSVD}, inner products between vectors of $\bm{U}_{\rm{s}}$ and $\bm{U}_{\rm{i}}$, especially for the first 20 pairs, are emphasized in \Cref{fig:brain_traveling}a. It can be seen that pairs of approximate opposite numbers are shown between both sides of the diagonal line, which is interpreted as a traveling wave based on \Cref{pro6.1}. In particular, the inner product between $\bm{U}_{\rm{s}}(:,1)$ and $\bm{U}_{\rm{i}}(:,2)$ is -0.8643, and between $\bm{U}_{\rm{s}}(:,2)$ and $\bm{U}_{\rm{i}}(:,1)$ is 0.9075, which makes the standard part of the sum of the first two PCs forms an obvious traveling wave. Making use of this discovery, the rank-2 matrix generated by these first two PCs is reshaped to match the original size, and then projected onto the brain applying "BrainNet" in MATLAB \cite{BrainNet}. 
After combining these brain image frames in series, a video is produced, and the change of color blocks reflects that of signals and brain activities. 
The prominent traveling process of signals is extracted in \Cref{fig:brain_traveling}b, and one of the blocks of the math task is underway during this time. 
For both hemispheres, it can be observed that the traveling waves start in Brodmann area 39 (angular gyrus) and end in Brodmann area 40 (supramarginal gyrus) \cite{Brodmann}. More precisely, these two Brodmann areas are both parts of Wernicke's area participating in language comprehension as well as language production, specifically the comprehension of speech sounds, which exactly matches the language task. On the other side, the traveling wave covered in the left hemisphere is more pronounced in contrast, which validates that the left hemisphere is the main processing system of language.

Subsequently, although the first two PCs contribute the vast majority of the original data matrix and demonstrate inner products with relatively large absolute values, the ninth and tenth PCs lead to the maximum and minimum inner products, which are $0.9495$ and $-0.8751$ as seen in \Cref{fig:brain_traveling}a. Likewise, these two PCs generate a rank-2 matrix and then can be restored as a sequential video of brain signals. Two kinds of traveling waves in the left hemisphere are obviously indicated in \Cref{fig:brain_traveling}c. One (see (i) and (ii)) is bi-directionally propagating between the Brodmann areas 22 and 39, where area 22 is the superior temporal gyrus (part of Wernicke’s area) and involves complex language and auditory processing. The other (see (iii) and (iv)) is between Brodmann areas 44, 45 and 6, where areas 44 (pars opercularis) and 45 (pars triangularis) are both parts of Broca's area associated with the praxis of speech. In addition, area 6 (premotor cortex) gets involved in language, including language processing and switching as well as syntactical processing \cite{BrodmannLanguageRelatedAreas}.

Accordingly, we summarize the three kinds of traveling waves described above and then introduce corresponding pathways in \Cref{fig:brain_traveling}d, where the green pathway between the Brodmann area 6 and 44, 45 is termed as Broca's stream due to its coverage areas. Besides, the red and blue pathways both propagate across Wernicke's area, thereby, they have named Wernicke's dorsal and ventral streams respectively in view of their relative positions in the brain. These results are in agreement with previous findings using structural connectivity techniques \cite{BrodmannLanguageRelatedAreas,LanguagePathway} and support the integrative strength of Broca's area and Wernicke's area in terms of language.

\section{Conclusions}\label{sec:conclusions}
In this paper, we propose the compact dual singular value decomposition (CDSVD) of dual complex matrices along with the necessary and sufficient conditions for its existence as well as its explicit expressions and convenient algorithms. For the CDSVD, other essential properties of dual matrices are investigated. To illustrate the significance of infinitesimal parts of dual matrices, we apply the CDSVD to time-series data and reveal how the data evolves over time, in particular, leading to traveling wave identification.

A notable theoretical result of our study is providing new insights into the SVD of dual matrices. Throughout the existing theoretical proofs and algorithms, they either require dual matrices with particular structures (e.g., full-column rank \cite{pennestri2018moore}), resulting in redundancy and costs (e.g., solving a redundant linear system involving Kronecker products \cite{pennestri2018moore}) or increase the complexity (e.g., using eigenvalue decomposition of one's Gram matrix to calculate its SVD \cite{pennestri2007linear,qi2022low}). Nevertheless, our CDSVD addresses the above problems and achieves better results.  We further present the rank-$k$ approximation based on the CDSVD after defining the rank of and metrics for dual matrices. In addition, we provide the existing conditions and explicit expression of the dual Moore-Penrose generalized inverse, which coincide with recent results \cite{wang2022dual}.

Another core finding of our study is how to deal with time-series data, identify intrinsic traveling waves and then extract their trajectories. We start with constructing a dual matrix using the original sequential data. The standard part is a real matrix with sampling points as rows and time points as columns, and its derivative or first-order difference is the infinitesimal part. The second step is to explore which two principal components generate traveling waves after calculating the CDSVD. Inner products reaching the maximum and minimum values between standard parts and infinitesimal parts of singular vectors unearth the most obvious traveling wave. Moreover, only a rank-2 matrix is used to describe the propagating process over time which costs less but contains more. It is worth noting that traveling wave identification in the brain tfMRI data with language processing task extracts three kinds of streams, which are all explored in classical brain regions related to language.

In summary, we first provide proof at the element level and an effective algorithm without redundant structures for the CDSVD of dual complex matrices. Then we recognize traveling waves in the brain based on our proposed CDSVD and verify the functional brain regions associated with the language task from our theoretical results and the existing literature. It remains to be elucidated the direction of our future research. Although our proposed CDSVD possesses a variety of prominent properties theoretically and practically, it still leaves room for improvement. For example, we can explore the case where zero singular values participate in the perturbation when the infinitesimal part of a dual matrix is considered as a perturbation of its standard part. Moreover, our method can be further used to validate or explore other streams, which will be of great advance to scholars in related research. 

\section*{Acknowledgments}
Data (in \cref{sec:brain}) were provided by the Human Connectome Project, WU-Minn Consortium (Principal Investigators: David Van Essen and Kamil Ugurbil; 1U54MH091657) funded by the 16 NIH Institutes and Centers that support the NIH Blueprint for Neuroscience Research; and by the McDonnell Center for Systems Neuroscience at Washington University.

The authors would like to thank the handling editor Eric K.-w. Chu and two anonymous reviewers for their valuable suggestions on our article.

\bibliographystyle{siamplain}
\bibliography{references}
\end{document}